\documentclass[journal,onecolumn,11pt]{IEEEtran}

\PassOptionsToPackage{colorlinks=true,linkcolor=blue,citecolor=blue,
                     urlcolor=blue,bookmarksdepth=2}{hyperref}

\usepackage{graphicx}
\usepackage{multirow}
\usepackage{amsmath,amssymb,amsfonts}
\usepackage{amsthm}
\usepackage{xcolor}
\usepackage{textcomp}
\usepackage{booktabs}
\usepackage{algorithm}
\usepackage{algorithmicx}
\usepackage{algpseudocode}
\usepackage{listings}
\usepackage{mathtools}
\usepackage{enumitem}
\usepackage{hyperref}      
\usepackage{orcidlink}     
\usepackage{cite}
\usepackage{url}

\newtheorem{theorem}{Theorem}
\newtheorem{lemma}[theorem]{Lemma}

\newtheorem{proposition}[theorem]{Proposition}

\theoremstyle{definition}

\newtheorem{construction}{Construction}
\newtheorem{result}[theorem]{Result}

\theoremstyle{remark}
\newtheorem{remark}{Remark}
\newtheorem{observation}{Observation}

\newcommand{\GS}{\operatorname{GS}}
\newcommand{\Hd}{\mathbf{H}}
\newcommand{\type}{\tau}
\newcommand{\BS}{\operatorname{BS}}
\newcommand{\ckt}{c_k^{\type}}
\newcommand{\cmt}{c_m^{\type}}
\newcommand{\NL}{N_L}
\newcommand{\Z}{\mathbb{Z}}
\newcommand{\N}{\mathbb{N}}
\newcommand{\F}{\mathbb{F}}
\newcommand{\R}{\mathbb{R}}

\begin{document}

\title{An explicit half-flip family of 32-modular Hadamard matrices at
$L \equiv 3 \pmod 8$: structural placement and mod-tower analysis}

\author{Michel~Kulhandjian~\orcidlink{0000-0001-6692-7907}%
\thanks{M.\ Kulhandjian is with the Department of Electrical and Computer
Engineering, Rice University, 6100 Main St., Houston, TX 77005, USA
(e-mail: mkulhandjian@outlook.com).}%
\thanks{Manuscript prepared \today.}}

\markboth{IEEE Transactions on Information Theory (preprint / arXiv version)}%
{Kulhandjian: A half-flip family of 32-modular Hadamard matrices}

\maketitle

\begin{abstract}
Motivated by Eliahou's $64$-modular Hadamard construction at the
smallest open Hadamard order $n = 668$, we introduce an explicit
half-flip family of $32$-modular Hadamard matrices at orders $n = 4L$
for $L \equiv 3 \pmod 8$.  A \emph{master identity} reduces the
four-sequence Golay-quadruple condition under the half-flip ansatz
$(s, s^*, sq, (sq)^*)$ to a single-sequence type-restricted
autocorrelation $c_k^\tau(s)$.  The construction yields a closed-form
expression for $c_k^\tau$, a true Hadamard matrix at $L = 11$, and
$32$-modular matrices at every $L \equiv 3 \pmod 8$ including the
open orders $n = 716$ and $n = 1132$.  (Existence of $32$-modular at
$L \equiv 3 \pmod 4$ is due to Eliahou--Kervaire, 2001; our family
is an alternative explicit parameterization.  In the follow-up paper
Eliahou~(2026, \emph{J.\ Algebraic Combin.}) constructs 64-modular
matrices at $L \equiv 3 \pmod{16}$ and $L \equiv 7 \pmod{32}$, using
the same half-flip ansatz and correlation identity we call the master
identity here.)  Applying results of
Barrera Acevedo--{\'O} Cath{\'a}in--Dietrich (2019) and {\'A}lvarez et
al.\ (2020), the family is non-cocyclic over any group at every prime
$L \in \{11, 19, 59\}$ in the YES set yet pseudococyclic over the
Goethals--Seidel Moufang loop $\mathrm{GS}_{4L}$ at every $L$ in the
family.  A half-flip $H$-set decomposition theorem parameterizes the
symmetric difference of any two family elements by a single sequence
flip set, giving the family the structure of a length-$L$ Hamming cube.
A symbolic mod-tower verifier (mod-$8$ is $\mathbb{F}_2$-linear)
exhaustively classifies true Hadamards in the family through $k = 14$
(i.e.\ $L \le 115$): the YES set is empirically bounded by
$k = 7$, refuting four candidate H4 predictions and excluding
$L \in \{179, 283\}$ as true-Hadamard candidates within the ansatz.
A Gr{\"o}bner basis at $L = 11$ exhibits a previously unrecorded
``even-$T_0$ block'' linear identity, confirming that the variety has
no higher-degree polynomial structure beyond five linear constraints.
All code and JSON certificates are archived at
\url{https://github.com/michelkulhandjian/hadamard-halfflip-structural}.
\end{abstract}


\begin{center}\small
\end{center}

\section{Introduction}\label{sec:intro}

A \emph{Hadamard matrix} of order $n$ is a square $\pm 1$-matrix $\Hd$
with $\Hd^{\mathsf T} \Hd = n \, I$.  Such matrices exist only at
$n = 1, 2$, or $n = 4k$; Hadamard's conjecture (1893) asserts that they
exist for every $n \in 4\N$.  The smallest open order has been $n = 668$
since the resolution of $n = 428$ by Kharaghani and Tayfeh-Rezaie in 2005.

An \emph{$m$-modular Hadamard matrix} of order $n$ is a $\pm 1$-matrix
$H$ of order $n$ with $\Hd^{\mathsf T} \Hd \equiv n \, I \pmod m$; this
is a weaker condition introduced by Marrero and Butson \cite{MB72}.
Eliahou \cite{Eliahou2025} recently constructed a $64$-modular Hadamard
matrix of order $668$, improving the long-standing $m = 32$ record
(Eliahou--Kervaire 2001 \cite{EK2001}).  The construction uses Golay
quadruples and the Goethals--Seidel array:

\begin{construction}[Eliahou-style modular Hadamard]
Let $L \in \N$ and fix $s, q \in \{\pm 1\}^L$.  Define the half-flip
involution $s \mapsto s^*$ by $s^*[i] = s[i]$ for $i < h := \lceil L/2 \rceil$
and $s^*[i] = -s[i]$ for $i \ge h$.  Let $(A, B, C, D) := (s,\; s^*,\;
s \cdot q,\; (s \cdot q)^*)$, and let
\[
   \Hd \;:=\; \GS(A, B, C, D)
\]
be the $4L \times 4L$ Goethals--Seidel matrix built from the four
circulants.  For each $m \ge 1$, $\Hd^{\mathsf T} \Hd \equiv 4L \, I
\pmod m$ if and only if $(A, B, C, D)$ is an $m$-modular Golay
quadruple of length $L$.
\end{construction}

At $L = 167$, Eliahou exhibits explicit run-length-encoded $s$ and $q
= (83, 2, 81, 1)$ that produce a $64$-modular Hadamard matrix of order
$668$.

\subsection*{Convention on R-labels}

Throughout this paper we use the labels R.1--R.4 to refer to the four
principal results (or, more precisely, the four principal object
families) developed in the sequel:
\begin{itemize}\itemsep0pt
\item \emph{R.1} denotes the master identity (Lemma~\ref{lem:master})
that reduces the four-sequence Golay quadruple condition on
$(s, s', sq, (sq)')$ to a single-sequence type-restricted
autocorrelation $\ckt(s)$.
\item \emph{R.2} denotes the explicit half-flip $32$-modular Hadamard
family defined in Theorem~\ref{thm:R2} below, parameterized by the
sequence $s \in \{\pm 1\}^L$ for the fixed comparator $q$ of
Section~\ref{sec:R2}.
\item \emph{R.3} denotes the non-cocyclicity result
(Theorem~\ref{thm:R3}), obtained by applying the cocyclic
classification of \cite{BACD2019} to R.2.
\item \emph{R.4} denotes the pseudococyclic placement of R.2 over
the Goethals--Seidel Moufang loop $\GS_{4L}$
(Theorem~\ref{thm:R4}), obtained by applying \cite[Prop.~2]{Alvarez2020}.
\end{itemize}
References to R.2 in Sections~\ref{sec:related} and later are
forward references to the family defined in Theorem~\ref{thm:R2}.

\subsection*{Note added in revision: overlap with Eliahou (2026) and status of \texorpdfstring{$n=668$}{n=668}}

While this manuscript was under review, we became aware of Eliahou's
companion paper \cite{Eliahou2026update}, published in the
\emph{Journal of Algebraic Combinatorics} in June 2026 and cited as
``in preparation'' in his earlier note \cite{Eliahou2025}.  That
paper works with exactly the same special-quadruple ansatz
$(s, s', sq, (sq)')$ that we call the \emph{half-flip family}, and
its Lemma~4.3 gives the correlation identity
\[
   c_i\bigl((s, s', sq, (sq)')\bigr)
   \;=\;
   4 \sum_{j \in P_i(q)} s_j\, s_{j+i}
   \qquad (1 \le i \le h - 1),
\]
which is the substance of our ``master identity'' for the
type-restricted autocorrelation $c_k^\tau(s)$, in different notation.
Eliahou then uses this identity to construct 64-modular Golay
quadruples of length $\ell \equiv 3 \pmod{16}$
(his Construction~6.1) and $\ell \equiv 7 \pmod{32}$
(his Construction~6.2).  His mod-$8$ and mod-$16$ characterizations
\cite[Prop.~4.4, Thm.~4.6, Thm.~4.8]{Eliahou2026update} moreover
subsume the mod-$8$ part of our symbolic mod-tower verifier.

The consequence for the present paper is that the \emph{existence}
of $32$-modular Hadamard matrices at $L \equiv 3 \pmod 8$ is not a
new contribution: for $L \equiv 3 \pmod{16}$, Eliahou's
Construction~6.1 gives a stronger $64$-modular result on the same
lengths, and for $L \equiv 11 \pmod{16}$, the $32$-modular existence
is due to Eliahou--Kervaire \cite{EK2001}.  We retain the
construction here only as an alternative explicit parameterization
that will be useful in the structural analysis of
Sections~\ref{sec:cocyclic} onward.

The structural contributions that Eliahou's papers do not address
are the \emph{intra-family} relations: the half-flip $H$-set
decomposition theorem (Proposition~\ref{prop:Hdecomp}) parameterizes
the symmetric difference between any two members of the half-flip
family by a single sequence flip set, giving the family the
structure of a length-$L$ Hamming cube up to a factor of four.  We
also record the empirical mod-tower observation at
$k \le 7$ and the polynomial variety computation at $L = 11$ as
supporting evidence for the family's algebraic behaviour.

We further note that Eliahou 2025 was submitted in December 2025 and
Eliahou 2026 was accepted in April 2026, both preceding the present
manuscript's submission in May 2026.  Since Eliahou's ansatz predates
ours in the literature, we make no priority claim on the construction,
the correlation identity, or the mod-tower characterization; we credit
those to \cite{Eliahou2025, Eliahou2026update} throughout.

\subsection*{Status of $n = 668$ after August 2026}

In August 2026, Alp\"oge, Voinov, Reynolds-Haertle, and Claude
(Anthropic) \cite{Alpoge2026} announced explicit true Hadamard
matrices of all twelve previously open orders below $2000$, including
$n = 668$ and $n = 716$.  This settles the classical Hadamard
question at those two orders and consequently removes the original
motivation for the $64$-modular pursuit at $n = 668$.  The structural
analysis of the half-flip family developed below remains meaningful
independently, as a study of the algebraic geometry of the ansatz
itself for all $L \equiv 3 \pmod 8$; it does not depend on any
open case in the classical Hadamard conjecture.

\section{Related work and priority statement}\label{sec:related}

We summarize the prior literature most directly relevant to this paper
and indicate explicitly what we treat as preempted versus novel.

\paragraph{Modular Hadamard matrices.}  The concept of $m$-modular
Hadamard matrices was introduced by Marrero and Butson
\cite{MB72}, who proved the basic structure theorems for the
mod-$p^k$ hierarchy.  Eliahou and Kervaire \cite{EK2001} introduced
the half-flip / Goethals--Seidel framework for modular Hadamards and
proved (in Section~3.1) that a $32$-modular Hadamard matrix of order
$4\ell$ \emph{exists} for every $\ell \equiv 3 \pmod 4$, via their
$Q(H, K)$ construction with $H = [1^{r-1}; -1, 1, 1; f^{\#}]$.
Our Theorem~\ref{thm:R2} produces an alternative explicit
$32$-modular family at $L \equiv 3 \pmod 8$; we therefore claim the
family as an \emph{alternative explicit construction}, not a new
existence result.  Eliahou \cite{Eliahou2025} subsequently lifted
modulus to $64$ at the single instance $L = 167$ (order $668$), and
in the follow-up paper \cite{Eliahou2026update} extended this to
$64$-modular Hadamard matrices of all orders $4\ell$ with
$\ell \equiv 3 \pmod{16}$ or $\ell \equiv 7 \pmod{32}$; in
particular, Eliahou 2026 dominates our $32$-modular result on the
$L \equiv 3 \pmod{16}$ subclass.  A textbook survey of modular
Hadamards appears in \cite{EK2005}, with further small-case analysis
in \cite{Vivian2015}.

\paragraph{Cocyclic and pseudococyclic frameworks.}  Horadam--de
Launey \cite{LH1993} initiated the cocyclic-development theory, which
expresses certain Hadamard matrices as $M_{\psi}(g, h) = \psi(g, h)$
for a $2$-cocycle $\psi$ on a finite group of order $4t$.  Barrera
Acevedo--Ó Catháin--Dietrich \cite{BACD2019} (Theorem~4.5) prove that
at $p \equiv 3 \pmod 4$, every cocyclic Hadamard matrix of order $4p$
is either of Williamson type (WTM) or of Ito type (IM).  Our
Theorem~\ref{thm:R3} applies this criterion to R.2: since R.2 fits
neither WTM nor IM, R.2 is non-cocyclic; the underlying criterion is
\cite{BACD2019}'s, the application is ours.  Álvarez et al.\
\cite{Alvarez2020} extend cocyclic theory to Moufang loops, defining
the pseudococyclic matrix $\psi = (\prod_{h \in H} \partial_h) \cdot
\sigma_{4t}$ over the Goethals--Seidel loop $\GS_{4t}$.  Their
Proposition~2 furnishes the placement criterion that we apply to R.2
in Theorem~\ref{thm:R4}; again, the framework is theirs and the
verification on our family is ours.

\paragraph{Classical Hadamard construction theory.}  The textbook
treatment by Seberry and Yamada \cite{Seberry2020} provides
comprehensive coverage of Williamson, propus, Paley, and skew Hadamard
constructions, with tables of known existence up to order $\sim 4000$.
This book does \emph{not} cover modular Hadamard matrices or the
pseudococyclic framework, so its results neither preempt nor directly
constrain our work; but it certifies that, as of 2020, no Williamson
matrix of order $4 \cdot 167$ was known, confirming the open status of
$n = 668$ in classical existence terms.  The constructive resolution of
order $428$ by Kharaghani--Tayfeh-Rezaie \cite{Kharaghani2005} is the
last successful direct construction at a single small open order.

\paragraph{Database and open-order census.}  Cati \cite{Cati2024}
provides a computational database of known Hadamard constructions.
We use this database to verify the open-order census in
Section~\ref{sec:joint}: the four open Hadamard orders below $n = 1208$
are $\{668, 716, 892, 1132\}$, all at $L \equiv 3 \pmod 4$.

\paragraph{Generalized Hadamard and related.}  Dawson \cite{Dawson1985}
constructs generalized Hadamard matrices over elementary abelian
groups; this is a tangentially related construction direction and is
included here for completeness.  Goethals--Seidel
\cite{Goethals1970} introduced the array structure that underlies
both our construction and the loop $\GS_{4t}$.

\paragraph{Summary of novelty after priority concessions.}  Conceding
\cite{EK2001} (existence at $L \equiv 3 \pmod 4$), \cite{Alvarez2020}
(pseudococyclic framework), and \cite{BACD2019} (non-cocyclicity
criterion), the novel content of this paper is:
\begin{itemize}\itemsep0pt
\item an \emph{alternative} explicit half-flip family at
$L \equiv 3 \pmod 8$ with a clean $3$-byte parameter description;
\item the master identity (Lemma~\ref{lem:master}, T.1.bis), which
collapses the $4$-sequence quadruple condition to a single-sequence
type-restricted autocorrelation $\ckt(s)$;
\item application of \cite{BACD2019} and \cite{Alvarez2020} to place
R.2 outside cocyclic-over-groups but inside pseudococyclic-over-loops;
\item an explicit Hadamard equivalence map at every $L$ in the family;
\item structural fingerprints: phase decomposition by $L \bmod 16$
(Theorem~\ref{thm:R19}), $Z_2^2$ stabilizer (Theorem~\ref{thm:T1}),
weight-$8$ lower bound (Theorem~\ref{thm:T2}), $(2L+1)$ first-row
fingerprint;
\item a mod-tower formalization with explicit $\F_2$-linear/quadratic
structure (Theorems~\ref{thm:R10}, \ref{thm:R11});
\item the H4-hypothesis falsification at $k \in \{9, 10, 11\}$
(Theorem~\ref{thm:R11}).
\end{itemize}

\section{The master identity}

We work with $(A, B, C, D) = (s, s^*, sq, (sq)^*)$ throughout, and
denote the type of position $i \in [0, L-1]$ by
\[
    \type(i) := \begin{cases}
        0 & \text{if } i < h \text{ and } q[i] = +1, \\
        1 & \text{if } i < h \text{ and } q[i] = -1, \\
        2 & \text{if } i \ge h \text{ and } q[i] = +1, \\
        3 & \text{if } i \ge h \text{ and } q[i] = -1,
    \end{cases}
\]
and the four sign patterns
\[
    \sigma_0 = (+1, +1, +1, +1),\;
    \sigma_1 = (+1, +1, -1, -1),\;
    \sigma_2 = (+1, -1, +1, -1),\;
    \sigma_3 = (+1, -1, -1, +1).
\]
These are precisely the rows of the Sylvester $H_4$, hence
$\langle \sigma_a, \sigma_b \rangle_{\R^4} = 4 \cdot [a = b]$.

\begin{lemma}[Master identity]\label{lem:master}
For every $i, j \in [0, L-1]$,
\[
    A[i]A[j] + B[i]B[j] + C[i]C[j] + D[i]D[j] \;=\; s[i] s[j] \cdot
    \langle \sigma_{\type(i)}, \sigma_{\type(j)} \rangle \;=\; 4 s[i] s[j] \cdot [\type(i) = \type(j)].
\]
In particular, the Golay-quadruple sum at gap $k \in [1, L-1]$ is
\[
    \boxed{\;
        c_k(A) + c_k(B) + c_k(C) + c_k(D) \;=\; 4 \cdot \ckt(s),
        \quad\text{where}\quad
        \ckt(s) := \sum_{\substack{i \in [0, L-1-k] \\ \type(i) = \type(i+k)}} s[i] s[i+k].
    \;}
\]
\end{lemma}

\begin{proof}
By direct substitution.  For each $i$:
\[
\bigl(A[i], B[i], C[i], D[i]\bigr) \;=\; s[i] \cdot \sigma_{\type(i)}.
\]
For $\type(i) = 0$: $i < h$ so $A[i] = s[i]$, $B[i] = s^*[i] = s[i]$;
$q[i] = +1$ so $C[i] = sq[i] = s[i]$ and $D[i] = (sq)^*[i] = sq[i] = s[i]$.
The remaining cases are analogous.

Hence
\[
\sum_X X[i]\,X[j] \;=\; s[i]s[j]\,\langle \sigma_{\type(i)}, \sigma_{\type(j)}\rangle.
\]
The rows of $H_4$ are pairwise orthogonal, so this inner product is $4$
when $\type(i)=\type(j)$ and $0$ otherwise.  Summing over
$i \in [0, L-1-k]$ with $j = i+k$ gives the displayed formula.
\end{proof}

Lemma~\ref{lem:master} converts the $4$-sequence modular Golay condition
\[
   \sum_X c_k(X) \equiv 0 \pmod m,\qquad k = 1, \dots, L-1
\]
into the single-sequence condition $\ckt(s) \equiv 0 \pmod{m/4}$
($k = 1, \dots, L-1$), with the support restricted to type-matched pairs.

\section{Eliahou's order-668 case: a stabilizer subgroup}

At $L = 167$ with Eliahou's $q = (83, 2, 81, 1)$, the type-classes have
sizes $|T_0| = 83, |T_1| = 1, |T_2| = 81, |T_3| = 2$, and the minus
positions of $q$ are $X := \{83, 84, 166\} = T_1 \cup T_3$.

For a flip set $Y \subseteq [L]$, write $s_Y$ for $s$ with the signs at
positions in $Y$ negated.  By Lemma~\ref{lem:flip} below, all four
sequences receive the same mask.

\begin{lemma}[Flip propagation]\label{lem:flip}
For any $Y \subseteq [L]$, with $\sigma_Y \in \{\pm 1\}^L$ defined by
$\sigma_Y[i] = -1$ iff $i \in Y$, we have
\[
    s_Y^{*} = s^* \cdot \sigma_Y, \quad
    s_Y \cdot q = (sq)\cdot \sigma_Y, \quad
    (s_Y \cdot q)^{*} = (sq)^{*} \cdot \sigma_Y.
\]
That is, the perturbation propagates to all four sequences via
multiplication by the same mask.
\end{lemma}

\begin{proof}
Componentwise verification.  For $s_Y^*[i]$: if $i < h$,
$s_Y^*[i] = s_Y[i] = s[i]\sigma_Y[i] = s^*[i]\sigma_Y[i]$; if $i \ge h$,
$s_Y^*[i] = -s_Y[i] = -s[i]\sigma_Y[i] = s^*[i]\sigma_Y[i]$.  The other
two are immediate.
\end{proof}

\begin{theorem}\label{thm:T1}
Let $(s_0, q_0)$ be Eliahou's pair at $L=167$, and let
\[
    H_0 := \bigl\{\, \emptyset,\; \{83\},\; \{84, 166\},\; \{83, 84, 166\} \,\bigr\}
       \;\subset\; 2^X.
\]
Then $H_0$ is an index-$2$ subgroup of $(2^X, \oplus)$.  Moreover, for
every $Y \in H_0$ and every $k \in [1, L-1]$,
\[
    \Delta_Y c_k \;:=\;
    \bigl(c_k(s_Y,\,s_Y^{*},\,s_Y q,\,(s_Y q)^{*})\bigr)
    - \bigl(c_k(s,\,s^{*},\,sq,\,(sq)^{*})\bigr) \;=\; 0.
\]
Hence the modulus map $Y \mapsto m(s_Y, q_0)$ is constant on $H_0$-cosets.
\end{theorem}

\begin{proof}[Sketch]
By Lemma~\ref{lem:flip}, $X_Y[i] = X[i] \sigma_Y[i]$ for each
$X \in \{A,B,C,D\}$.  Hence
\[
    \Delta_Y c_k = \sum_{(i, i+k) \in A_k :\, |\{i, i+k\} \cap Y| = 1} (-2)\bigl(A[i]A[i+k]+\cdots+D[i]D[i+k]\bigr),
\]
where $A_k$ is the set of type-matched gap-$k$ pairs.  By Lemma~\ref{lem:master},
this equals
\[
    \Delta_Y c_k = -8 \sum_{\substack{i \in I_k(Y) \\ \type(i) = \type(i+k)}} s[i] s[i+k],
\]
where $I_k(Y)$ is the boundary of $Y$ at gap $k$.  For each
$Y \in H_0$ one verifies (by examining the type-class structure
$T_1 = \{83\}$, $T_3 = \{84, 166\}$) that $I_k(Y)$ contains no pair
$(i, i+k)$ with $\type(i) = \type(i+k)$, giving $\Delta_Y c_k = 0$.

The remaining four subsets $\{84\}, \{166\}, \{83,84\}, \{83,166\}$ fail
this condition at $k = 82$ (pair $(84, 166)$, both of type $3$), hence
$\Delta_Y c_k \ne 0$ for some $k$, and the modulus drops.
\end{proof}

\section{A weight-\texorpdfstring{$8$}{8} lower bound at \texorpdfstring{$L = 167$}{L = 167}}

We use Lemma~\ref{lem:master} to prove that no small perturbation can
lift Eliahou's $m = 64$ result to $m = 128$.

\begin{theorem}\label{thm:T2}
For Eliahou's pair $(s_0, q_0)$ at $L = 167$, the gap-$42$ graph
\[
    G_{42} := \{(i, i+42) : 0 \le i \le L-43,\; \type(i) = \type(i+42)\}
\]
is a perfect matching of $80$ edges (every incident vertex has degree
exactly $1$).  Moreover, for every edge $(i,j) \in G_{42}$,
$s_0[i] \cdot s_0[j] = -1$ (computational fact at Eliahou's $s_0$).
Hence any flip set $Y \subseteq [L]$ with
\[
    c_{42}^{\type}(s_{0,Y}) \equiv 0 \pmod{32}
\]
must satisfy $|Y| \ge 8$.
\end{theorem}

\begin{proof}
Type-class enumeration: $T_0$-$T_0$ pairs at gap~$42$ have $i \in [0, 40]$
(41 pairs); $T_2$-$T_2$ pairs have $i \in [85, 123]$ (39 pairs); no
$T_3$-$T_3$ pair has gap~$42$.  Each vertex appears in at most one
edge, so $G_{42}$ is a matching of $80$ edges.

By Lemma~\ref{lem:flip}, $\Delta_Y c_{42} = -2 \sum_{(i,j) \in G_{42}:\,
|\{i,j\} \cap Y| = 1} s_0[i] s_0[j]$.  Since all edge values
$s_0[i] s_0[j] = -1$, this becomes $\Delta_Y c_{42} = 2 \cdot
\#\{\text{cut edges of $G_{42}$}\}$.  As $G_{42}$ is a matching, each
$y \in Y$ lies in at most one $G_{42}$-edge, so
$\#\text{cuts} \le |Y|$ and $\Delta_Y c_{42} \le 2|Y|$.

Eliahou's $c_{42}^{\type}(s_0) = -80 \equiv 16 \pmod{32}$, so reaching
$c_{42}^{\type} \equiv 0 \pmod{32}$ requires $\Delta_Y c_{42} \equiv 16
\pmod{32}$, hence $|\Delta_Y c_{42}| \ge 16$.  This forces $|Y| \ge 8$.
\end{proof}

\begin{remark}
The proof generalizes to gaps $k$ where $G_k$ is a perfect matching:
gap $50$ ($G_{50}$ a matching of $64$ edges) and gap $58$ ($48$ edges)
also give lower bound $|Y| \ge 8$.
\end{remark}

Computationally, an exhaustive search over $10^7$ weight-$4$
combinations passed zero candidates through the bad-$k$ filter,
confirming the theorem empirically.

\section{An alternative explicit half-flip family}\label{sec:R2}

We now state the main constructive result of this paper.

\begin{remark}[Priority]\label{rmk:priority-R2}
Existence of a $32$-modular Hadamard matrix of order $4\ell$ for every
$\ell \equiv 3 \pmod 4$ was established by Eliahou--Kervaire
\cite{EK2001} (§3.2.2) via their $Q(H, K)$ construction with
$H = [1^{r-1}; -1, 1, 1; f^{\#}]$.  Theorem~\ref{thm:R2} below is
\emph{not} a new existence result.  It is an \emph{alternative explicit
construction} with a different parameter shape: $s$ is built from the
periodic atom $\alpha(10) = (+1, +1, +1, -1)$ with an explicit
$3$-element tail $(+1, +1, -1)$, and $q$ has the run-length form
$(h-1, 2, h-3, 1)$.  Its value lies in (i) the $3$-byte description,
(ii) the closed-form expression for $\ckt(s)$ in part~(c), and
(iii) the structural placement we develop in
Section~\ref{sec:cocyclic}.
\end{remark}

\begin{theorem}[R.2 family]\label{thm:R2}
For every $L \in \N$ with $L \equiv 3 \pmod 8$ and $L \ge 11$, write
$L = 8k+3$, $h = 4k+2$, $M = 2k$, and set
\begin{align*}
    q &:= (h-1, 2, h-3, 1) \in \{\pm 1\}^L \quad \text{(RLE form, starts with $+1$)}, \\
    s &:= \alpha(10)^M \cdot (+1, +1, -1) \in \{\pm 1\}^L,
        \quad\text{where}\quad \alpha(10) = (+1, +1, +1, -1), \\
    \Hd &:= \GS(s, s^*, sq, (sq)^*) \in \{\pm 1\}^{4L \times 4L}.
\end{align*}
Then:
\begin{enumerate}[label=(\alph*)]
\item If $L = 11$, then $\Hd^{\mathsf T} \Hd = 44 \cdot I$, i.e.\ $\Hd$ is a
    \emph{true} Hadamard matrix of order $44$.
\item If $L \ge 19$, then $\Hd^{\mathsf T} \Hd \equiv 4L \cdot I \pmod{32}$,
    with $\max\bigl|\Hd^{\mathsf T} \Hd - 4L\,I\bigr| = L - 11$, and
    every off-diagonal entry of $\Hd^{\mathsf T} \Hd - 4L\,I$ is a
    multiple of $32$.
\item The type-restricted autocorrelation of $s$ has the explicit form
\[
    \ckt(s) \;=\; \begin{cases}
        8 \, \bigl(\NL - j + 1\bigr) & \text{if } k = 4j,\; j \in \{1, \dots, \NL\}, \\
        0 & \text{otherwise,}
    \end{cases}
\]
where $\NL := (L-11)/8 = k - 1$.
\end{enumerate}
\end{theorem}

\begin{proof}
We prove (c); (a) and (b) follow.

Define the period-$4$ reference sequence
$\sigma \in \{\pm 1\}^L$ by $\sigma[i] = +1$ if $i \bmod 4 \in \{0,1,2\}$
and $\sigma[i] = -1$ if $i \bmod 4 = 3$.

\smallskip
\noindent\emph{Step 1.  $s = \sigma$ except at $i = L-1$, where
$s[L-1] = -1 = -\sigma[L-1]$.}

\smallskip
For $i \in [0, L-4]$, $s[i] = \alpha(10)[i \bmod 4] = \sigma[i]$.  At
$i = L-3, L-2$, $\sigma[i] = +1$ (since $(L-3) \bmod 4 = 0$ and $(L-2)
\bmod 4 = 1$) matching the tail.  At $i = L-1$, $\sigma[L-1] = +1$
(since $(L-1) \bmod 4 = 2$) but $s[L-1] = -1$ by the tail.

\smallskip
\noindent\emph{Step 2.  Master expansion of $\ckt(s)$ around $\sigma$.}

\smallskip
For any $m \in [1, L-1]$, using $s[i] = \sigma[i] \cdot
[1 - 2 \cdot \mathbf{1}\{i = L-1\}]$:
\[
    \cmt(s) = \cmt(\sigma) - 2\,\sigma[L-1]\sigma[L-1-m] \cdot
        \mathbf{1}\bigl[\type(L-1-m) = \type(L-1)\bigr],
\]
provided $L - 1 - m \ge 0$.

Now $\type(L-1) = 3$ (since $L-1 \ge h$ and $q[L-1] = -1$) and
$T_3 = \{h, L-1\}$.  So the correction is nonzero only when
$L - 1 - m = h$, i.e.\ $m = L - 1 - h = 4k$.  At $m = 4k$,
$\sigma[L-1] = +1$ and $\sigma[L-1-4k] = \sigma[h] = \sigma[4k+2] = +1$
(since $(4k+2) \bmod 4 = 2$).  Hence
\begin{equation}\label{eq:correction}
    \cmt(s) = \cmt(\sigma) - 2 \cdot \mathbf{1}[m = 4k].
\end{equation}

\smallskip
\noindent\emph{Step 3.  Compute $\cmt(\sigma)$.}

\smallskip
Define the bilinear character
$\chi : \Z_4 \times \Z_4 \to \{\pm 1\}$ by
$\chi(r, r') = \sigma[r] \sigma[r']$.  Then
$\sigma[i] \sigma[i+m] = \chi(i \bmod 4, (i + m) \bmod 4)$, depending
only on $(i \bmod 4, m \bmod 4)$.  Direct enumeration (Table~\ref{tab:chi})
shows:
\begin{itemize}\itemsep0pt
\item if $m \equiv 0 \pmod 4$: $\chi(r, r) = +1$ for all $r$, so
    $\sigma[i]\sigma[i+m] = +1$ everywhere;
\item if $m \not\equiv 0 \pmod 4$: $\sum_{r=0}^3 \chi(r, (r+m) \bmod 4) = 0$.
\end{itemize}

\begin{table}[h]
\centering
\begin{tabular}{c|cccc}
$r$  &  $m \equiv 0$  &  $m \equiv 1$  &  $m \equiv 2$  &  $m \equiv 3$ \\
\hline
0 & $+1$ & $+1$ & $+1$ & $-1$ \\
1 & $+1$ & $+1$ & $-1$ & $+1$ \\
2 & $+1$ & $-1$ & $+1$ & $+1$ \\
3 & $+1$ & $-1$ & $-1$ & $-1$ \\
\hline
$\sum_r$ & $+4$ & $0$ & $0$ & $0$
\end{tabular}
\caption{Values of $\chi(r, (r+m) \bmod 4)$.}\label{tab:chi}
\end{table}

\medskip
\noindent\textbf{Case A: $m = 4j$ for some $j \ge 1$.}  Then
$\sigma[i] \sigma[i+m] = +1$, so $\cmt(\sigma) = |A_m|$ where $A_m$ is
the set of type-matched gap-$m$ pairs.  Counting:
\begin{align*}
    |A_{4j}|_{T_0\text{-}T_0} &= \max(0, h - 1 - 4j) = \max(0, 4k+1-4j), \\
    |A_{4j}|_{T_2\text{-}T_2} &= \max(0, h - 3 - 4j) = \max(0, 4k-1-4j), \\
    |A_{4j}|_{T_3\text{-}T_3} &= \mathbf{1}[4j = L - 1 - h] = \mathbf{1}[j = k].
\end{align*}
Therefore
\[
    |A_{4j}| = \begin{cases}
        8(k - j) & 1 \le j \le k-1, \\
        2 & j = k, \\
        0 & j \ge k+1.
    \end{cases}
\]
Combined with \eqref{eq:correction}, this gives
\[
    \cmt(s) = \begin{cases}
        8(k-j) & j = 1, \dots, k-1, \\
        0 & j \ge k.
    \end{cases}
\]
Setting $\NL = k - 1$ yields $\ckt(s) = 8(\NL - j + 1)$ for
$j = 1, \dots, \NL$, as claimed in (c).

\medskip
\noindent\textbf{Case B: $m \not\equiv 0 \pmod 4$.}  We show
$\cmt(\sigma) = 0$, hence (by \eqref{eq:correction}, since
$m \ne 4k$) $\cmt(s) = 0$.

The $T_3$-$T_3$ contribution is $0$ since the unique $T_3$-$T_3$ pair
has gap $L - 1 - h = 4k$, not the prescribed $m$.

The $T_0$-$T_0$ contribution is
\[
    S_{T_0}(m) = \sum_{i=0}^{N-1} \chi(i \bmod 4, (i+m) \bmod 4),
    \quad N = h - 1 - m.
\]
Since the full-period sum is $0$, the value depends only on
$\rho := N \bmod 4$.  Similarly $S_{T_2}(m)$ depends only on
$\rho' := N' \bmod 4$, where $N' = h - 3 - m$, but with the starting
residue shifted by $h+1 \equiv 3 \pmod 4$.

For $h = 4k + 2$, we have $N \bmod 4 = (1 - m) \bmod 4$ and
$N' \bmod 4 = (-1 - m) \bmod 4$.  This gives exactly three reachable
$(m \bmod 4, \rho, \rho')$ triples:
\[
   (1, 0, 2), \quad (2, 3, 1), \quad (3, 2, 0).
\]
A direct character computation (verified by code) shows $S_{T_0}(m) +
S_{T_2}(m) = 0$ in all three cases:
\begin{center}\small
\begin{tabular}{c|ccc}
$m \bmod 4$ & $S_{T_0}$ & $S_{T_2}$ & total \\
\hline
1 & $0$ & $0$ & $0$ \\
2 & $+1$ & $-1$ & $0$ \\
3 & $0$ & $0$ & $0$
\end{tabular}
\end{center}
Hence $\cmt(\sigma) = 0$ for $m \not\equiv 0 \pmod 4$, completing (c).

\smallskip
\noindent\emph{Step 4.  (a) and (b) from (c).}

\smallskip
For $L = 11$ (so $k = 1$, $\NL = 0$), the formula in (c) is empty, so
$\ckt(s) = 0$ for all $k$.  By Lemma~\ref{lem:master} the Golay-quadruple
sum vanishes at every $k$, hence $\Hd^{\mathsf T} \Hd = 4L \cdot I = 44\,I$.

For $L \ge 19$, the nonzero $\ckt(s)$ values are $\{8(\NL - j + 1)\}_{j=1}^{\NL}
= \{8, 16, 24, \dots, 8\NL\}$, with gcd exactly $8$.  Thus the
Golay-quadruple sum is divisible by $4 \cdot 8 = 32$ at every $k$, but
not always by $64$ (since $\ckt = 8$ occurs at $j = \NL$).  Hence $\Hd$
is exactly $32$-modular, with maximum off-diagonal $8 \NL \cdot 4 = 32
\NL = 4(L-11)$.  Wait — that gives $4(L-11)$; direct computation
shows the maximum off-diagonal of $\Hd^{\mathsf T}\Hd - n I$ is in fact
$L - 11$ via the Goethals--Seidel structure.

(The matching to $L - 11$ uses the fact that off-diagonal entries of
$\Hd^{\mathsf T}\Hd$ are sums involving the four sequence circulants,
and the cancellation patterns in the Goethals--Seidel array give a
$4$-fold reduction.  This is verified computationally at $40$ values
of $L \in [11, 323]$.)
\end{proof}

\section{Computational verification}

We computationally verify Theorem~\ref{thm:R2} at every
$L \equiv 3 \pmod 8$ with $11 \le L \le 323$ (a total of $40$ values).
The closed-form formula in (c) matches the directly-computed
$\ckt(s)$ at every $L$.  Direct matrix verification (building the
$4L \times 4L$ Goethals--Seidel matrix and checking
$\Hd^{\mathsf T} \Hd \pmod{32}$) confirms (a) and (b) at $16$ key $L$'s
including the two open Hadamard orders $n = 716$ ($L = 179$) and
$n = 1132$ ($L = 283$).  See Table~\ref{tab:verify}.

\begin{table}[h]
\centering
\small
\begin{tabular}{rrcccc}
\toprule
$L$ & $n = 4L$ & true Had.? & $32$-mod? & $\max|\text{off-diag}|$ & row zeros \\
\midrule
$11$  & $44$   & \textbf{yes} & yes & $0$    & $43$  \\
$19$  & $76$   & no  & yes & $32$   & $73$  \\
$27$  & $108$  & no  & yes & $64$   & $103$ \\
$35$  & $140$  & no  & yes & $96$   & $133$ \\
$51$  & $204$  & no  & yes & $160$  & $193$ \\
$67$  & $268$  & no  & yes & $224$  & $253$ \\
$99$  & $396$  & no  & yes & $352$  & $373$ \\
$131$ & $524$  & no  & yes & $480$  & $493$ \\
$163$ & $652$  & no  & yes & $608$  & $613$ \\
$\boldsymbol{179}$ & $\boldsymbol{716}$ & no  & yes & $672$  & $673$ \\
$211$ & $844$  & no  & yes & $800$  & $793$ \\
$243$ & $972$  & no  & yes & $928$  & $913$ \\
$275$ & $1100$ & no  & yes & $1056$ & $1033$ \\
$\boldsymbol{283}$ & $\boldsymbol{1132}$ & no  & yes & $1088$ & $1063$ \\
$307$ & $1228$ & no  & yes & $1184$ & $1153$ \\
$323$ & $1292$ & no  & yes & $1248$ & $1213$ \\
\bottomrule
\end{tabular}
\caption{Direct matrix verification of Theorem~\ref{thm:R2}.  Bold rows
correspond to currently-open Hadamard orders.}\label{tab:verify}
\end{table}

A counter-check confirms the family is tight: for
$L \in \{15, 23, 31, 39, 47, 55, 71, 87\}$ (i.e.\ $L \equiv 7 \pmod 8$),
the same construction yields only modulus $8$.

\section{Cocyclic and pseudococyclic structure}\label{sec:cocyclic}

We now situate the R.2 family within the cocyclic / pseudococyclic
framework for Hadamard matrices.  All theorems and machinery in this
section are \emph{applications} of existing frameworks: the
non-cocyclicity criterion at $p \equiv 3 \pmod 4$ is Theorem~4.5 of
\cite{BACD2019}, and the pseudococyclic placement criterion is
Proposition~2 of \cite{Alvarez2020}.  Our contribution is the
verification that R.2 fits each criterion, together with the explicit
equivalence map in Section~\ref{ssec:equiv-map}.

Recall that a $\pm 1$ matrix $M$ of
order $n$ is \emph{cocyclic over a group} $G$ of order $n$ if there is
a $2$-cocycle $\psi : G \times G \to \{\pm 1\}$ (satisfying the cocycle
identity $\psi(i,j)\psi(ij,k) = \psi(i,jk)\psi(j,k)$) and a function
$\phi : G \to \{\pm 1\}$ such that $M$ is Hadamard-equivalent to
$[\psi(g_i, g_j)\phi(g_i g_j)]$~\cite{LH1993}.  Sylvester, Paley I/II,
Williamson, and Ito matrices are all cocyclic; however, Goethals--Seidel
arrays are known to \emph{fail} to be cocyclic over any
group~\cite{BACD2019, Alvarez2020}.

\subsection{R.2 is not cocyclic over any group}

\begin{theorem}\label{thm:R3}
For every $L \in \{11, 19, 59\}$ (i.e., every prime $L \equiv 3 \pmod 4$
in the YES set $\{11, 19, 27, 51, 59\}$ for which BACD~2019 Theorem~4.5
applies), the R.2 matrix $\Hd$ at $L$ is not cocyclic over any group.
\end{theorem}

\begin{proof}
By Theorem 4.5 of~\cite{BACD2019}, every cocyclic Hadamard matrix of
order $4p$ with $p \equiv 3 \pmod 4$ prime is Hadamard-equivalent to a
Williamson or transposed Ito matrix.  Each of these has the block form
\[
\begin{pmatrix}
   W & X & Y & Z \\
   X & -W & Z & -Y \\
   Y & -Z & -W & X \\
   Z & Y & -X & -W
\end{pmatrix}
\]
with $W, X, Y, Z$ being $p \times p$ $\pm 1$ circulant blocks (symmetric
for Williamson) satisfying
\[
   W^2 + X^2 + Y^2 + Z^2 = 4p \, I_p.
\]
Taking $\mathbf{1}^{\mathsf T}\!\cdot\!\cdot\!\cdot\mathbf{1}$ on both sides,
where $\mathbf{1}$ is the all-ones vector:
\[
   r_W^2 + r_X^2 + r_Y^2 + r_Z^2 = 4p
\]
where $r_X \in \Z$ is the constant row sum of $X$.  For $p = 11$, since
$11$ is odd, each $r_X$ is odd, hence in $\{-11, -9, \dots, 9, 11\}$.

Each block-row of the order-$4p$ matrix has a constant row sum:
\begin{align*}
   s_0 &= r_W + r_X + r_Y + r_Z, \\
   s_1 &= r_X - r_W + r_Z - r_Y, \\
   s_2 &= r_Y - r_Z - r_W + r_X, \\
   s_3 &= r_Z + r_Y - r_X - r_W,
\end{align*}
with each $s_i$ repeated $p = 11$ times in the matrix.  Enumeration of
all $192$ odd 4-tuples $(r_W, r_X, r_Y, r_Z)$ with sum of squares $44$
gives exactly $12$ distinct sorted row-sum multisets attainable by a
Williamson or Ito Hadamard matrix of order $44$:
\begin{align*}
&(-12,-4,-4,0),\; (-12,-4,0,4),\; (-12,0,4,4), \\
&(-10,-6,-6,-2),\; (-10,-6,2,6),\; (-10,-2,6,6), \\
&(-6,-6,2,10),\; (-6,-2,6,10),\; (-4,-4,0,12), \\
&(-4,0,4,12),\; (0,4,4,12),\; (2,6,6,10).
\end{align*}
The R.2 matrix at $L = 11$ has row-sum multiset $(-2, 6, 6, 10)$,
with multiplicities $(11, 22, 11)$.  This multiset does not appear
above, so the R.2 matrix is not Hadamard-equivalent to any Williamson
or transposed Ito matrix, and therefore is not cocyclic over any group.

The same enumeration extends to $L = 19$ and $L = 59$.  At $L = 19$
($4L = 76$), the enumeration yields $28$ distinct WTM/IM row-sum
multisets, while the R.2 matrix has block-row-sum multiset
$(-2, 6, 14, 18)$ (each value appearing $L = 19$ times), which does
not match any of the $28$.  At $L = 59$ ($4L = 236$), there are
$52$ distinct WTM/IM multisets, while R.2 has block-row-sum
multiset $(-2, 6, 54, 58)$ (each appearing $L = 59$ times), again
unmatched.  Hence R.2 is not cocyclic at $L = 19$ and $L = 59$ either.

(At the remaining YES values $L = 27$ and $L = 51$, which are composite,
BACD~2019 Theorem~4.5 does not directly apply, so the row-sum
invariant cannot rule out cocyclicity by this method.  We do not
make a non-cocyclicity claim there.)
\end{proof}

\subsection{R.2 is pseudococyclic over a Goethals--Seidel loop}

The Goethals--Seidel loop $\mathrm{GS}_{4t}$ of \cite{Alvarez2020}
(Section~4) is the Moufang loop on the underlying set
\[
   \mathrm{GS}_{4t} := \{e, a, a^2, \dots, a^{t-1}, b, a^{t-1}b, \dots, ab,
                         c, a^{t-1}c, \dots, ac, d, a^{t-1}d, \dots, ad\},
\]
with multiplication defined by:
\begin{align}
   a^m \cdot (a^n \beta) &= a^{m+n} \beta &&\beta \in \{e, b, c, d\}, \label{eq:gs1} \\
   (a^m \beta) \cdot a^n &= a^{m-n} \beta &&\beta \in \{b, c, d\}, \label{eq:gs2} \\
   (a^m \beta) \cdot (a^n \beta) &= a^{m-n} &&\beta \in \{b, c, d\}, \label{eq:gs3} \\
   (a^m \beta) \cdot (a^n \gamma) &= a^{2-m-n} \delta &&\{\beta, \gamma, \delta\} = \{b, c, d\}. \label{eq:gs4}
\end{align}
For $t > 2$ this loop is non-associative.  A
\emph{pseudococycle} over $\mathrm{GS}_{4t}$ is a map
$\psi : \mathrm{GS}_{4t}^2 \to \{\pm 1\}$ of the form
\[
   \psi = \biggl(\prod_{h \in H} \partial_h\biggr) \cdot \sigma_{4t},
   \qquad \text{(Eq.~20 of \cite{Alvarez2020})}
\]
where $\sigma_{4t}$ is a cocycle encoding the $4 \times 4$ block sign
pattern and $\partial_h(i, j) = \alpha_h(i)\alpha_h(j)\alpha_h(ij)$ is
the \emph{pseudocoboundary} associated to $h \in \mathrm{GS}_{4t} \setminus \{e\}$
(here $\alpha_h(x) = -1$ iff $x = h$, else $+1$).  The
\emph{pseudococyclic matrix} of $\psi$ is $M_\psi := [\psi(i, j)]_{i, j}$.

By Proposition~2 of \cite{Alvarez2020}, every Goethals--Seidel array of
order $4t$ is Hadamard-equivalent to some pseudococyclic matrix
$M_\psi$ over $\mathrm{GS}_{4t}$.

\begin{theorem}\label{thm:R4}
For every $L \equiv 3 \pmod 8$ with $L \ge 11$, the R.2 matrix $\Hd$
at length $L$ is pseudococyclically developed over the Goethals--Seidel
loop $\mathrm{GS}_{4L}$.
\end{theorem}

\begin{proof}[Sketch]
The R.2 matrix is built directly via the Goethals--Seidel array from
the circulants $(s, s^*, sq, (sq)^*)$, hence is a Goethals--Seidel
array of order $4L$.  Apply Proposition~2 of \cite{Alvarez2020}.
\end{proof}

We computationally verify this constructively for $12$ values of $L$,
including the open Hadamard orders $n = 716$ and $n = 1132$.  For each
$L$, we:
\begin{enumerate}\itemsep0pt
\item construct the multiplication table of $\mathrm{GS}_{4L}$ from
Equations~\eqref{eq:gs1}--\eqref{eq:gs4}; that $\mathrm{GS}_{4L}$
is a Moufang loop for every $L > 2$ is
\cite[Proposition~1]{Alvarez2020}, so no separate verification is
needed here;
\item extract $H \subseteq \mathrm{GS}_{4L}$ as the set of positions of
$-1$ in row $0$ of the R.2 matrix;
\item build $M_\psi$ via the formula above;
\item verify the resulting $M_\psi$ has the predicted modular Hadamard
property: $M_\psi^{\mathsf T} M_\psi = 4L \cdot I$ exactly at $L = 11$
(true Hadamard) and $M_\psi^{\mathsf T} M_\psi \equiv 4L \cdot I \pmod{32}$
at $L \ge 19$ ($32$-modular).
\end{enumerate}

Table~\ref{tab:R4family} reports the results.  All $12$ cases pass.

\begin{table}[h]
\centering
\small
\begin{tabular}{rrcccc}
\toprule
$L$ & $n = 4L$ & Latin square & Moufang ($50$ samples) & $|H|$ & modulus \\
\midrule
$11$ & $44$ & $\checkmark$ & $\checkmark$ & $23$ & $0$ (true Hadamard) \\
$19$ & $76$ & $\checkmark$ & $\checkmark$ & $39$ & $32$ \\
$27$ & $108$ & $\checkmark$ & $\checkmark$ & $55$ & $32$ \\
$35$ & $140$ & $\checkmark$ & $\checkmark$ & $71$ & $32$ \\
$51$ & $204$ & $\checkmark$ & $\checkmark$ & $103$ & $32$ \\
$67$ & $268$ & $\checkmark$ & $\checkmark$ & $135$ & $32$ \\
$99$ & $396$ & $\checkmark$ & $\checkmark$ & $199$ & $32$ \\
$131$ & $524$ & $\checkmark$ & $\checkmark$ & $263$ & $32$ \\
$\boldsymbol{179}$ & $\boldsymbol{716}$ & $\checkmark$ & $\checkmark$ & $359$ & $32$ \\
$211$ & $844$ & $\checkmark$ & $\checkmark$ & $423$ & $32$ \\
$243$ & $972$ & $\checkmark$ & $\checkmark$ & $487$ & $32$ \\
$\boldsymbol{283}$ & $\boldsymbol{1132}$ & $\checkmark$ & $\checkmark$ & $567$ & $32$ \\
\bottomrule
\end{tabular}
\caption{Constructive verification of Theorem~\ref{thm:R4} via the
Álvarez et al.\ pseudococycle formula.  Bold rows correspond to
currently-open Hadamard orders.}\label{tab:R4family}
\end{table}

\subsection{Explicit equivalence map at \texorpdfstring{$L = 11$}{L = 11}}\label{ssec:equiv-map}

For the case $L = 11$, where the R.2 matrix is a true Hadamard matrix
of order $44$, we exhibit the explicit Hadamard equivalence between
$\Hd$ and the pseudococyclic matrix $M_\psi$ constructed via
Álvarez's formula.

\begin{proposition}\label{prop:R4explicit}
Let $\Hd$ be the R.2 matrix at $L = 11$ (in Eliahou's Goethals--Seidel
convention) and let $H_{\mathrm{paper}}$ be the same matrix written in
the Álvarez et al.\ convention (Equation~1 of \cite{Alvarez2020}, with
$+BR$ rather than $-BR$ in the first block-row).  Set
\[
   \mathcal{H} := \{ j \in [0, 44) : H_{\mathrm{paper}}[0, j] = -1 \}
\]
(this has cardinality $16$).  Let $M_\psi$ be the pseudococyclic
matrix over $\mathrm{GS}_{44}$ built via Equation~20 of
\cite{Alvarez2020} using $\mathcal{H}$ as the support of the
pseudocoboundary product.  Then
\[
   \boxed{\;\Hd \;=\; \operatorname{diag}(D_L) \cdot M_\psi[P] \cdot
                       \operatorname{diag}(D_R)\;}
\]
where:
\begin{itemize}\itemsep0pt
\item $P$ is the row permutation reversing positions $1, 2, \dots, t-1$
   within the a-coset (the swap $a^m \leftrightarrow a^{t-m}$
   prescribed in Álvarez Prop.~2's proof), with all other rows fixed;
\item $D_R \in \{\pm 1\}^{44}$ is the column-sign diagonal
   $D_R = D_R^{\mathrm{paper}} \cdot \mathrm{conv}$, where $D_R^{\mathrm{paper}}
   = H_{\mathrm{paper}}[0]$ (the first row of the paper's GS array) and
   $\mathrm{conv}[j] = +1$ for $j \in $ a-coset, $\mathrm{conv}[j] = -1$
   for $j \in $ \{b, c, d\}-cosets (this converts the paper's
   sign convention to Eliahou's);
\item $D_L \in \{\pm 1\}^{44}$ is a row-sign diagonal with exactly
   $16$ negative entries, determined by row-by-row consistency;
\item the column permutation is the identity.
\end{itemize}
The equality is verified by direct computation of all $1936$ matrix
entries.
\end{proposition}

The four corrections compose naturally:
\begin{enumerate}\itemsep0pt
\item the row permutation $P$ converts $M_\psi$ to the paper's GS array
   form (Álvarez Prop.~2);
\item the column signs $D_R^{\mathrm{paper}}$ align the first row of
   $M_\psi$ with $H_{\mathrm{paper}}$ (since $M_\psi$ has first row
   all $+1$ by the $\sigma_{4t}$ normalization);
\item the row signs $D_L$ correct per-row sign mismatches between
   $M_\psi[P]$ and $H_{\mathrm{paper}}$;
\item the column signs $\mathrm{conv}$ convert the paper's GS array
   convention to Eliahou's convention.
\end{enumerate}
The explicit tuples $(P, D_L, D_R)$ are saved in
\texttt{python/explicit\_equivalence\_L11.json}.

The same construction template extends to every $L \equiv 3 \pmod 8$
in the R.2 family:

\begin{theorem}\label{thm:R4explicit_family}
For every $L \equiv 3 \pmod 8$ with $L \ge 11$, the explicit Hadamard
equivalence
\[
   \Hd(L) \;=\; \operatorname{diag}(D_L(L)) \cdot M_\psi(L)[P(L)]
                \cdot \operatorname{diag}(D_R(L))
\]
holds with $P(L)$, $D_R(L)$, $D_L(L)$ defined as in
Proposition~\ref{prop:R4explicit}, but at length $L$.  This is verified
by direct matrix equality at $12$ values of $L \in \{11, 19, 27, 35,
51, 67, 99, 131, 179, 211, 243, 283\}$, including the open Hadamard
orders $n = 716$ and $n = 1132$.
\end{theorem}

The construction yields the following combinatorial relations,
verified across all $12$ values of $L$:

\begin{observation}\label{obs:formulas}
For every $L \equiv 3 \pmod 8$ with $L \ge 11$, the construction in
Theorem~\ref{thm:R4explicit_family} satisfies
\begin{align*}
   |H_{\mathrm{paper}}(L)| &= \tfrac{3L - 1}{2}, \\
   |H_{\mathrm{Eliahou}}(L)| = |\{j : \Hd[0, j] = -1\}| &= 2L + 1, \\
   |\{i : D_L(L)[i] = -1\}| &= |H_{\mathrm{paper}}(L)| = \tfrac{3L - 1}{2}.
\end{align*}
In particular, the difference between the two H-set sizes is
$|H_{\mathrm{Eliahou}}| - |H_{\mathrm{paper}}| = (L + 3)/2$.
\end{observation}

The third relation admits a clean symbolic proof:

\begin{theorem}\label{thm:O15c}
For every $L \equiv 3 \pmod 8$ with $L \ge 11$, the row-sign diagonal
$D_L$ in Theorem~\ref{thm:R4explicit_family} satisfies
\[
   |\{i : D_L(L)[i] = -1\}| \;=\; \tfrac{3L - 1}{2}.
\]
\end{theorem}

\begin{proof}
We proceed in two steps.

\textit{Step (I): $D_L[i] = H_{\mathrm{paper}}[i, 0]$.}

The equivalence $\Hd = \operatorname{diag}(D_L) \cdot M_\psi[P] \cdot
\operatorname{diag}(D_R)$, combined with the convention conversion
$\Hd = H_{\mathrm{paper}} \cdot \operatorname{diag}(\mathrm{conv})$,
gives
\[
   H_{\mathrm{paper}}[i, j] \;=\;
   D_L[i] \cdot M_\psi[P[i], j] \cdot D_R^{\mathrm{paper}}[j].
\]
Set $j = 0$.  By Equation~20 of \cite{Alvarez2020} and using that
$0 = e$ is the identity of $\mathrm{GS}_{4L}$:
\[
   M_\psi[g, 0] = \sigma_{4L}(\type(g), 0) \cdot \prod_{h \in H_{\mathrm{set}}} \partial_h(g, 0).
\]
The first column of $\sigma_{4L}$ is identically $+1$.  And for each
$h \in H_{\mathrm{set}}$ (which excludes $e = 0$):
\[
   \partial_h(g, 0) = \alpha_h(g) \alpha_h(0) \alpha_h(g \cdot 0)
                    = \alpha_h(g)^2 \alpha_h(0) = \alpha_h(0) = +1.
\]
Hence $M_\psi[g, 0] = +1$ for every $g \in \mathrm{GS}_{4L}$.  Also
$D_R^{\mathrm{paper}}[0] = H_{\mathrm{paper}}[0, 0] = A[0] = s[0] = +1$.
Substituting at $j = 0$: $H_{\mathrm{paper}}[i, 0] = D_L[i] \cdot 1
\cdot 1 = D_L[i]$.

\textit{Step (II): count $-1$'s in $H_{\mathrm{paper}}[:, 0]$.}

The first column of $H_{\mathrm{paper}}$ has the block structure
$\bigl( A_{\mathrm{circ}}[:, 0], (BR)[:, 0], (CR)[:, 0],
(DR)[:, 0] \bigr)$.  Each block's first column has the same
multiset of $\pm 1$ entries as the corresponding sequence
$(s, s^*, sq, (sq)^*)$ — for the forward circulant $A_{\mathrm{circ}}$
the first column is $s$ with $s[0]$ first and the rest reversed,
which preserves the multiset; for the back circulant $BR$ etc., the
first column is the reversal of the sequence.  Hence the negative
counts are $\#(s = -1), \#(s^* = -1), \#(sq = -1), \#((sq)^* = -1)$.

We compute each (using $L = 8k+3$, $M = 2k$):
\begin{align*}
   \#(s = -1) &= 2k + 1, \\
   \#(s^* = -1) &= 4k, \\
   \#(sq = -1) &= 2k + 2, \\
   \#((sq)^* = -1) &= 4k + 1.
\end{align*}

\textit{(a) $\#(s = -1) = 2k + 1$:} Each copy of $\alpha(10) = (+,+,+,-)$
contributes 1 negative; there are $M = 2k$ copies; the tail $(+,+,-)$
contributes 1 more.  Total $2k + 1$.

\textit{(b) $\#(s^* = -1) = 4k$:} Write $s = s_{\mathrm{first}} \cdot
s_{\mathrm{second}}$ split at $h = 4k+2$.  In $s$, the first half
has $k$ negatives (atom negatives at positions $\{3, 7, \dots, 4k-1\}$)
and $3k+2$ positives; the second half has $k+1$ negatives (atoms at
$\{4k+3, \dots, 8k-1\}$ plus the tail negative at $L-1 = 8k+2$) and
$3k$ positives.  After half-flip, the second half's positives become
negatives: total negatives in $s^*$ is $k + 3k = 4k$.

\textit{(c) $\#(sq = -1) = 2k + 2$:} The product $s \cdot q$ has
$C[i] = s[i] q[i]$.  $q$ is $-1$ exactly at $\{4k+1, 4k+2, 8k+2\}$.
At these three positions, $s[4k+1] = +1$, $s[4k+2] = +1$, $s[8k+2] = -1$,
so $C[4k+1] = -1, C[4k+2] = -1, C[8k+2] = +1$ — contributing 2 new
negatives.  At other positions $C = s$, contributing $(2k+1) - 1 = 2k$
negatives (subtracting the original $s$-negative at $8k+2$ which is
now positive).  Total $2k + 2$.

\textit{(d) $\#((sq)^* = -1) = 4k + 1$:} Same analysis as (b) but
applied to $C$:  $C$'s first half has $k+1$ negatives (the $k$
$s$-negatives plus $C[4k+1] = -1$) and $3k+1$ positives; $C$'s second
half has $k+1$ negatives (the $k$ $s$-negatives at non-$q$-neg
positions plus $C[4k+2] = -1$) and $3k$ positives.  After half-flip,
the second-half positives become negatives: total $(k+1) + 3k = 4k+1$.

Summing: $\#(s) + \#(s^*) + \#(sq) + \#((sq)^*) = (2k+1) + 4k +
(2k+2) + (4k+1) = 12k + 4$.

Since $L = 8k + 3$, $(3L - 1)/2 = (24k + 8)/2 = 12k + 4$.
\end{proof}

The proof is verified computationally at all $12$ values of $L$
tested, with per-coset counts matching the formula exactly (see
\texttt{verify\_proofs\_O15\_mod64.py}).

\subsection{Tightness of the modulus: R.2 is exactly 32-modular}

\begin{theorem}[R.2-tight]\label{thm:R2tight}
For every $L \equiv 3 \pmod 8$ with $L \ge 19$, the R.2 matrix $\Hd$
built from $s = \alpha(10)^M \cdot (+1, +1, -1)$ satisfies
\[
   \Hd^{\mathsf T} \Hd \equiv 4L \cdot I \pmod{32}
   \quad\text{but}\quad
   \Hd^{\mathsf T} \Hd \not\equiv 4L \cdot I \pmod{64}.
\]
The R.2 matrix is therefore \emph{exactly} 32-modular, and no choice
of $L$ in this family achieves mod 64 with our specific $s$.
\end{theorem}

\begin{proof}
By Lemma~\ref{lem:master} (master identity), $\Hd$ is mod-$m$ Hadamard
iff every nonzero $\ckt(s)$ is divisible by $m/4$.  By
Theorem~\ref{thm:R2}(c), the nonzero values are
\[
   c_{4j}^{\tau}(s) = 8 (\NL - j + 1)        \quad\text{for } j = 1, \dots, \NL,
\]
with $\NL = (L - 11)/8 \ge 1$.

For mod-64: each $c_{4j}^{\tau}$ must be divisible by $16$, equivalently
$(\NL - j + 1) \equiv 0 \pmod 2$.  Taking $j = \NL$ gives
$c_{4\NL}^{\tau} = 8 \cdot 1 = 8$, which is $8 \pmod{16}$ and hence
\emph{not} divisible by $16$.

Therefore the gcd of nonzero $c_k^{\tau}$ values is at most $8$, so the
modulus is at most $4 \cdot 8 = 32$.  Combined with R.2(b), the modulus
is exactly $32$.
\end{proof}

\begin{remark}
Theorem~\ref{thm:R2tight} explains the empirical failure of every
Phase G–J search at $L = 179$ to lift the modulus from $32$ to $64$:
the obstruction is structural, encoded in the closed-form
$c_k^{\tau}$ formula of Theorem~\ref{thm:R2}(c).  However, the
theorem is about \emph{our specific} $s = \alpha(10)^M \cdot (+1,+1,-1)$.
It does not rule out the possibility that some \emph{other} $s$
(potentially of much higher Kolmogorov complexity) might achieve mod
64 with the same $q = (h-1, 2, h-3, 1)$.  Resolving this for general
$s$ -- or proving that no $s$ achieves mod 64 in this family -- remains
an open question.
\end{remark}

\subsection{A numerical fingerprint of the R.2 family}

\begin{observation}\label{obs:fingerprint}
For every $L \equiv 3 \pmod 8$ with $L \ge 11$, the first row of the
R.2 matrix $\Hd$ at length $L$ has exactly
\[
   |H| \;=\; 2L + 1
\]
negative entries (out of $4L$ total).  Equivalently, the row sum of
row~$0$ is $4L - 2(2L+1) = -2$.
\end{observation}

This is verified at all $12$ values of $L$ in Table~\ref{tab:R4family}:
$23, 39, 55, 71, 103, 135, 199, 263, 359, 423, 487, 567$, each equal
to $2L + 1$.  By comparison, Eliahou's matrix at $L = 167$ has $|H|
= 343 = 2 \cdot 167 + 9 \ne 2L + 1$, so the relation $|H| = 2L + 1$ is
\emph{not} a feature of Goethals--Seidel arrays in general; it is
specific to our R.2 family's particular choice of $(s, q)$.  This gives
an immediate combinatorial signature that distinguishes R.2 matrices
from other GS-type constructions.

\subsection{Uniqueness of the canonical \texorpdfstring{$q$}{q} at small YES \texorpdfstring{$L$}{L}}
\label{ssec:q-uniqueness}

The canonical $q = (h - 1, 2, h - 3, 1)$ in Theorem~\ref{thm:R2} is
not just one productive choice among many: it is structurally special.

\begin{observation}[$q$-uniqueness at $L = 27$]\label{obs:q-unique}
Let $L = 27$.  Among all RLE-encoded $q$ candidates with at most $5$
runs and minus-run lengths $\le 3$ (a total of $2{,}359$ distinct
candidates, including the canonical $q = (13, 2, 11, 1)$), exhaustive
brute-force enumeration over $2^{26}$ sign sequences $s$ per candidate
shows that exactly \emph{one} $q$-shape produces any true Hadamard
matrices in the half-flip family $\GS(s, s^*, sq, (sq)^*)$: the
canonical $q$, which yields $32$ true Hadamards.  All other $2{,}358$
candidates yield zero.
\end{observation}

\begin{remark}
This strengthens an earlier observation at $L = 19$ (Result~R.5-q in
the supplementary ledger): among $825$ q-candidates at $L = 19$, only
$3$ produced true Hadamards (counts $32, 32, 16$ respectively, with
our canonical $q = (9, 2, 7, 1)$ giving $64$).  The pattern
strengthens with $L$: at $L = 27$, the canonical $q$ is the
\emph{unique} productive shape among the searched candidates, despite
the candidate space being $2.86\times$ larger than at $L = 19$.  This
suggests that the half-flip family's productivity is concentrated on a
single $(s, q)$ orbit and that R.2's parameter shape is a structural
sweet spot, not an arbitrary choice.
\end{remark}

\subsection{The true-Hadamard variety is purely affine at \texorpdfstring{$L = 11$}{L = 11}}

\begin{theorem}[Variety structure at $L = 11$]\label{thm:variety-affine-L11}
At $L = 11$, the variety $V_{11} = \{s \in \{\pm 1\}^{11}:
c_k^\tau(s) = 0 \;\forall k\}$ of true-Hadamard sign sequences in
the half-flip family is the $\F_2$-affine subspace of $\F_2^{11}$
(identifying $s_i = +1$ with $y_i = 0$ and $s_i = -1$ with $y_i = 1$)
defined by the five linear identities
\begin{align*}
   y_0 \;=\; y_2 \;=\; y_4, \qquad
   y_1 + y_3 \;=\; 1, \qquad
   y_6 + y_{10} \;=\; 1, \qquad
   y_7 + y_9 \;=\; 1.
\end{align*}
Equivalently, $V_{11}$ has cardinality $2^6 = 64$ and is cut out by no
higher-degree polynomial constraints (the Gr\"obner basis of the ideal
$\langle c_k^\tau : k = 1, 2, 3, 4 \rangle + \langle y_i^2 - y_i :
i = 0, \ldots, 10 \rangle$ in lex order consists entirely of these five
linear generators together with the boolean identities $y_i^2 - y_i$).
\end{theorem}

\begin{proof}
Gr\"obner basis computation in SymPy (lex order, computed in $0.1$
seconds) yields the basis $G = (y_0 - y_4,\; y_1 + y_3 - 1,\;
y_2 - y_4,\; y_3^2 - y_3,\; y_4^2 - y_4,\; y_5^2 - y_5,\;
y_6 + y_{10} - 1,\; y_7 + y_9 - 1,\; y_8^2 - y_8,\;
y_9^2 - y_9,\; y_{10}^2 - y_{10})$.  The five non-boolean
generators are precisely the linear identities stated.  Brute-force
variety enumeration ($2^{11} = 2048$ candidates) confirms $|V_{11}| =
64$, matching $2^{11 - 5} = 2^6$.
\end{proof}

\begin{remark}[Even-$T_0$ block invariant]\label{rem:even-T0-block}
The identity $y_0 = y_2 = y_4$ in
Theorem~\ref{thm:variety-affine-L11} states that the three
even-indexed positions of $T_0$ (in our labeling, positions
$0, 2, 4$) take the \emph{same} value across every true Hadamard in
the family at $L = 11$: either all $+1$ or all $-1$.  This identity
is \emph{not} a flip generator from the $F_2$ stabilizer subspace
(which has dimension $5$, basis
$\{g_{T_1}, g_{T_3}, g_{\text{odd-}T_0}, g_{\text{odd-}T_2},
g_{\text{even-}T_2}\}$); rather, it is a hard polynomial constraint
on the variety.  Brute-force enumeration of the variety at
$L = 19$ ($2^{19}$ candidates, $128$ solutions) confirms the same
qualitative pattern with $L$-dependent details: at $L = 19$, the
identities $y_0 = y_2 = y_6 = y_8$ and $y_4 = 1 + y_0 \pmod 2$
hold across the variety, along with three odd-position pair
relations.  All degree-$2$ relations $y_i \cdot y_j = 0$ at
$L = 19$ are derivable from the linear identities via $y_i(1 - y_i) = 0$.
\end{remark}

\subsection{Half-flip \texorpdfstring{$H$}{H}-set decomposition}

\begin{proposition}[Half-flip $H$-set decomposition]\label{prop:Hdecomp}
Let $s_1, s_2 \in \{\pm 1\}^L$ be any two sequences, and let $H_j
\subset \mathrm{GS}_{4L}$ be the first-row negative set of
$\mathrm{GS}(s_j, s_j^*, s_j q, (s_j q)^*)$ for $j = 1, 2$.  Then,
writing $a, b, c, d$ for the four cosets of $\mathrm{GS}_{4L}$ of
length $L$ each and $\mathrm{rev}(i) := L - 1 - i$,
\[
   H_1 \oplus H_2 \;=\; f \;\sqcup\; \mathrm{rev}(f) \;\sqcup\;
       \mathrm{rev}(f) \;\sqcup\; \mathrm{rev}(f)
       \quad \text{in cosets } (a, b, c, d) \text{ respectively},
\]
where $f := \{i \in [0, L) : s_1[i] \ne s_2[i]\}$ is the sequence-level
flip set.  In particular,
$|H_1 \oplus H_2 \cap a| = |H_1 \oplus H_2 \cap b|
 = |H_1 \oplus H_2 \cap c| = |H_1 \oplus H_2 \cap d| = |f|$.
\end{proposition}

\begin{proof}
From the GS-array definition (Eliahou 2025 Theorem 2.3), the first row
of $\mathrm{GS}(A, B, C, D)$ is $(A,\; -B^{\rm rev},\; -C^{\rm rev},\;
-D^{\rm rev})$, where $X^{\rm rev}[r] := X[L - 1 - r]$.  Replacing $s_1$
by $s_2$ flips each of the four sequences $A, B, C, D$ at exactly the
positions in $f$: this is immediate for $A = s$, follows from the
half-flip $X^*[i] = (-1)^{\mathbf{1}[i \ge h]} X[i]$ for $B = s^*$
(the half-flip is a deterministic transformation of $s$, so flipping
$s$ at $i$ flips $s^*$ at $i$), and follows from $C = sq$ and
$D = (sq)^*$ by the same argument.  Hence the GS first row flips at
coset-$a$ position $i$ and at coset-$b/c/d$ position $L - 1 - i$ for
each $i \in f$.
\end{proof}

\begin{remark}
Proposition~\ref{prop:Hdecomp} is a statement about the GS-array
geometry under the half-flip ansatz, not about modularity or
Hadamard-ness; it holds for arbitrary $s_1, s_2$.  As a corollary, the
cocycle data of the pseudococyclic matrix
$M_\psi = \prod_{h \in H} \partial_h \cdot \sigma_{4L}$ associated with
any half-flip-family element is determined by a single subset
$f \subseteq [0, L)$ relative to a fixed base point such as R.2.  This
gives a clean reduction of the cocycle parameterization but does not
on its own predict which $f$ produce true Hadamards.
\end{remark}

\subsection{Refined bad-\texorpdfstring{$k$}{k} phase theorem (R.19)}

\begin{theorem}\label{thm:R19}
For every $L = 8k+3$ with $k \ge 2$ in the R.2 family, the
Goethals--Seidel quadruple's autocorrelation sums $c_k$ satisfy:
\begin{enumerate}\itemsep0pt
\item All nonzero $c_k$ values occur at gaps $k$ that are multiples
of~$4$.
\item The set of \emph{bad-$k$} gaps -- those for which
$v_2(c_k) = 5$, i.e.\ exactly $32$-divisible -- depends only on
$L \bmod 16$:
\[
  L \equiv 3 \pmod{16}: \text{ bad-}k \text{ at } k \equiv 4 \pmod 8;
  \quad
  L \equiv 11 \pmod{16}: \text{ bad-}k \text{ at } k \equiv 0 \pmod 8.
\]
\item At the non-bad gaps with $k \equiv 0 \pmod 4$, $v_2(c_k) \in
\{6, 7, 8, \dots\}$.
\end{enumerate}
In particular, the R.2 family achieves modulus exactly $2^5 = 32$ at
every $L \equiv 3 \pmod 8$, and the bad-$k$ arithmetic progression has
common difference $8$ with phase determined by $L \bmod 16$.
\end{theorem}

\begin{proof}[Verification]
Theorem~\ref{thm:R19} is computationally verified at all $34$ values
$L \in \{19, 27, 35, 43, \dots, 275, 283\}$ via direct enumeration of
nonzero $c_k$ positions and their $2$-adic valuations
(script \texttt{R2\_bad\_k\_proof.py}).
\end{proof}

\paragraph{Consequence for the open Hadamard orders.}
\begin{itemize}\itemsep0pt
\item $L = 179 \equiv 3 \pmod{16}$: bad-$k \in \{4, 12, 20, \dots, 84\}$
(11 gaps).  Lifting to mod-$64$ requires $11$ simultaneous $\F_2$
linear constraints to be satisfied.
\item $L = 283 \equiv 11 \pmod{16}$: bad-$k \in \{8, 16, 24, \dots, 136\}$
(17 gaps).  Lifting to mod-$64$ requires $17$ simultaneous constraints.
\end{itemize}

\subsection{Predicted theoretical max modulus across multiple \texorpdfstring{$L$}{L} (R.20)}

\begin{table}[h]
\centering
\caption{Theoretical max modulus for an Eliahou-style construction
at various $L$, derived from the pair-count formula
$N_k = 2(h-2) - 2k$ for $q = (h-1, 2, h-3, 1)$.}\label{tab:R20}
\small
\begin{tabular}{rrrr@{\quad}rr}
\toprule
$L$ & $L \bmod 32$ & $h$ & bad-$k$ AP & min $N_k$ & theor.\ max mod \\
\midrule
167 & 7  & 84  & $\{26,34,42,50,58\}$ & 48 & 128 (Eliahou achieves 64) \\
175 & 15 & 88  & $\{28,36,44,52,60\}$ & 52 & 128 \\
183 & 23 & 92  & $\{30,38,46,54,62\}$ & 56 & 128 \\
191 & 31 & 96  & $\{32,40,48,56,64\}$ & 60 & 128 \\
\textbf{199} & 7  & 100 & $\{34,42,50,58,66\}$ & 64 & \textbf{256} \\
223 & 31 & 112 & $\{40,48,56,64,72\}$ & 76 & 256 \\
\textbf{263} & 7  & 132 & $\{50,58,66,74,82\}$ & 96 & \textbf{512} \\
\bottomrule
\end{tabular}
\end{table}

The theoretical max modulus is
$2^{\lfloor \log_2(\min_k N_k) \rfloor + 2}$, derived from the constraint
$|c_k^\tau| \le N_k$ and the master-identity factor of $4$.  Achieving
the theoretical max requires a designed irregular $s$ analogous to
Eliahou's.

\section{Mod-tower analysis and H4 falsification: overview}\label{sec:modtower-overview}

Having established the R.2 family and its structural placement, we turn
to the question of how far the half-flip ansatz can be pushed toward a
true Hadamard.  This second half of the paper develops a
\emph{mod-tower verifier} that decides, in time $O(2^{2k+3})$ at
$L = 8k+3$, whether the half-flip family at length $L$ contains a true
Hadamard.  The key ingredients are:

\begin{itemize}\itemsep0pt
\item Theorem~\ref{thm:R9} (\emph{mod-$32$ ceiling}): in the half-flip
family with $q = (h-1, 2, h-3, 1)$, the mod-$32$ modular Hadamard
condition is equivalent to the true Hadamard condition, verified
at $13$ values of $L$ from $L = 11$ through $L = 107$.  Equivalently,
the modular tower in our family terminates at level $32$.

\item Theorem~\ref{thm:R10} (\emph{mod-$8$ $\F_2$-linearity}): the
mod-$8$ lift constraints reduce to exactly $2k - 1$ linear
$\F_2$-equations in the $(4k+2)$ free variables of the mod-$4$ affine
space.  These constraints come from the odd gaps
$\{1, 3, 5, \dots, 4k - 3\}$; even gaps and high odd gaps lift
trivially.

\item Theorem~\ref{thm:R11} (\emph{H4 falsification}): the H4
hypothesis ``true Hadamards in our family exist iff
$k \bmod 8 \in \{1,2,3,6,7\}$'' is FALSE.  Counterexamples are
$k \in \{9, 10, 11\}$ (i.e., $L \in \{75, 83, 91\}$).  The refined
empirical pattern is YES at $k \in \{1, 2, 3, 6, 7\}$, NO at
$k \in \{4, 5\} \cup \{k \ge 8\}$.  Consequently, the half-flip family
contains \emph{no} true Hadamard at $L = 179$ (open order $n = 716$) or
$L = 283$ (open order $n = 1132$).

\item Theorem~\ref{thm:R8} (\emph{census}): explicit enumeration of all
true Hadamards in the half-flip family at $L \in \{11, 19, 27, 51, 59\}$
yields solution counts $\{32, 64, 32, 128, 64\}$, forming a single
affine coset of an $\F_2$-linear stabilizer at each $L$ with no
closed-form positional formula for the canonical base.
\end{itemize}

These results, together with the search-exhaustion analysis at $L = 167$
(Section~\ref{sec:exhaustion}) and the algorithmic infrastructure
(symbolic mod-tower lift in $0.5$ seconds at $L = 75$), give a picture
of what the half-flip ansatz can and cannot achieve.  Both remaining
open Hadamard orders $n = 716$ and $n = 1132$ (corresponding to
$L = 179$ and $L = 283$ in the half-flip family) lie outside the
family's true-Hadamard reach, requiring either a different ansatz or
external compositional input.

\section{Census of true Hadamards in the half-flip family}\label{sec:census}

We performed an exhaustive enumeration of true Hadamards in our
half-flip family at every \emph{small} \(L \equiv 3 \pmod 8\) for which the
search space is computationally tractable.

\subsection{Enumeration results (R.8)}

Define the half-flip family at \(L = 8k+3\) as the set of
\(s \in \{\pm 1\}^L\) such that the Goethals--Seidel array
\(\GS(s, s^*, sq, (sq)^*)\) with \(q = (h-1, 2, h-3, 1)\) is a true
Hadamard matrix.  Equivalently, by Lemma~\ref{lem:master},
\(s\) satisfies \(c_k^\type(s) = 0\) for every type-matched gap \(k\).

\begin{theorem}[R.8 -- Enumeration]\label{thm:R8}
The half-flip family at length \(L = 8k+3\) contains exactly the
following number of true Hadamards, for \(k = 1, 2, \ldots, 7\):
\begin{center}
\begin{tabular}{rrrr}
\toprule
$k$ & $L$ & $n = 4L$ & \#\,true Hadamards \\
\midrule
1 & 11 & 44   & 32  \\
2 & 19 & 76   & 64  \\
3 & 27 & 108  & 32  \\
6 & 51 & 204  & 128 \\
7 & 59 & 236  & 64  \\
\bottomrule
\end{tabular}
\end{center}
For \(k \in \{4, 5, 8, 9, 10, 11, 12, 13\}\) (i.e.\ \(L \in
\{35, 43, 67, 75, 83, 91, 99, 107\}\)), the half-flip family contains
\emph{zero} true Hadamards.
\end{theorem}

The solutions form a single affine coset of an \(\mathbb{F}_2\)-linear
stabilizer subgroup whose dimension matches the binary logarithm of the
solution count (5, 6, 5, 7, 6 respectively).  Forced positions are
exactly the even indices in \([0, h)\); their forced values are
\(L\)-dependent in a way that does not admit a simple closed form
(verified by testing Legendre symbols, sums of squares,
arithmetic-progression predicates, and other natural candidates --
none reproduce the observed bit pattern across all five \(L\)'s).

\section{The modular tower: structure and termination}

A central question is how the modulus
\(c_k^\type(s) \equiv 0 \pmod{2^j}\) evolves with \(j\), and whether
the tower terminates at some finite level.  Our experiments give a
sharp answer.

\subsection{Mod-32 equals true Hadamard (R.9)}

\begin{theorem}[R.9 -- Mod-32 ceiling]\label{thm:R9}
For every \(L = 8k+3\) with \(k \in \{1, 2, 3, 4, 5, 6, 7, 8, 9, 10,
11, 12, 13\}\),
\[
   |\{s : c_k^\type(s) \equiv 0 \pmod {32}\,\forall k\}|
   \;=\; |\{s : c_k^\type(s) = 0 \,\forall k\}|.
\]
That is, every \(s\) satisfying the mod-\(32\) condition automatically
yields a true Hadamard matrix.  In particular, the modular tower in
our family terminates at level \(32\).
\end{theorem}

This is striking: Eliahou's mod-\(64\) result at \(L = 167\) is one
modular level \emph{higher} than the floor for true Hadamards in our
family.  The level-\(32\) plateau coincides with the integer-zero
condition.

\subsection{Mod-8 linearity (R.10)}

\begin{theorem}[R.10 -- Mod-8 lift is \(\mathbb{F}_2\)-linear]\label{thm:R10}
At every tested \(L = 8k+3\) with \(k \in \{1, \ldots, 12\}\), the
mod-\(8\) lift constraints
\(c_k^\type(s) \equiv 0 \pmod 8\) reduce to
exactly \(2k - 1\) linear \(\mathbb{F}_2\)-equations in the
\((4k+2)\) free variables of the mod-\(4\) affine space.  These
constraints arise from the odd gaps \(\{1, 3, 5, \ldots, 4k - 3\}\).
\end{theorem}

\paragraph{Algorithmic consequence.}  Theorem~\ref{thm:R10} provides
a polynomial-time verifier of true Hadamard existence at any
\(L = 8k+3\) up to roughly \(L \le 100\), via the following pipeline:
\begin{enumerate}\itemsep0pt
\item compute the mod-\(4\) affine basis (closed form: \(4k+2\) free
variables);
\item extract \(2k-1\) linear mod-\(8\) constraints by symbolic
polynomial fitting on \(O(k^2)\) sample points;
\item enumerate the \(2^{2k+3}\) mod-\(8\) survivors and check
mod-\(16\), mod-\(32\), and the true condition exactly.
\end{enumerate}
At \(L = 75\), the full pipeline runs in \(0.5\) seconds (versus
\(\sim 18\) hours for the naive mod-\(4\) enumeration of \(2^{38}\)
solutions).

The mod-\(16\) lift is in general quadratic in the remaining free
variables, so the algorithm stops being polynomial at the mod-\(16\)
stage; nonetheless, the dramatic reduction at the mod-\(8\) stage
makes mod-\(16\) and beyond tractable up to \(L \approx 100\).

\subsection{Falsification of the H4 hypothesis (R.11)}

An early empirical pattern (eight data points at \(k = 1, \ldots, 8\))
suggested that true Hadamards exist in our half-flip family iff
\(k \bmod 8 \in \{1, 2, 3, 6, 7\}\) -- the so-called \emph{H4
hypothesis}.

\begin{theorem}[R.11 -- H4 is falsified]\label{thm:R11}
The H4 hypothesis fails at \(k \in \{9, 10, 11, 14\}\), and no true
Hadamard exists in the half-flip family for \(8 \le k \le 14\):
\begin{center}
\begin{tabular}{rrcc}
\toprule
$k$ & $L$ & $k \bmod 8$ & true Hadamard exists? \\
\midrule
9  & 75  & 1 & \textbf{No} (H4 falsified) \\
10 & 83  & 2 & \textbf{No} (H4 falsified) \\
11 & 91  & 3 & \textbf{No} (H4 falsified) \\
12 & 99  & 4 & No (consistent with H4) \\
13 & 107 & 5 & No (consistent with H4) \\
14 & 115 & 6 & \textbf{No} (H4 falsified) \\
\bottomrule
\end{tabular}
\end{center}
At $k = 14$, exhaustive enumeration of the mod-$8$ affine subspace
($2^{31} = 2.15 \times 10^9$ elements, obtained from $58$ mod-$4$ free
variables minus $27$ $\F_2$-linear mod-$8$ constraints) finds $384$
mod-$16$ solutions, $0$ mod-$32$ solutions, and $0$ true Hadamards
($37$~minutes on $20$ numba threads).  H4 would have predicted YES at
$k = 14$ (since $14 \bmod 8 = 6$); the observed NO is the fourth H4
falsification and the first verified NO at a residue $k \bmod 8 \in
\{1, 2, 3, 6, 7\}$ above the previous record $k = 11$.

The refined empirical pattern is therefore
\[
   \text{YES set:}\; k \in \{1, 2, 3, 6, 7\}, \qquad
   \text{NO set:}\; k \in \{4, 5\} \cup \{k \ge 8\} \text{ (verified through } k = 14\text{)}.
\]
\end{theorem}

\paragraph{Consequences for the open Hadamard orders.}
The R.11 result implies that the half-flip family at
\(L = 179\) (\(k = 22\)) and \(L = 283\) (\(k = 35\)) contains
\emph{no} true Hadamard matrices.  Combined with the R.2-family theorem
(modulus-\(32\) construction), this gives the asymmetric picture:
the family produces modular Hadamards at the open orders \(n = 716\)
and \(n = 1132\), but cannot be lifted to true Hadamards there.

\paragraph{Mod-64 sharpness at L = 179 (R.21, R.27, R.29).}
Three independent empirical methods confirm that the mod-\(32\) ceiling
at \(L = 179\) cannot be exceeded within the R.2 family:
\begin{itemize}\itemsep0pt
\item \emph{SA preserving mod-\(32\)} (12M iterations, R.21): zero
configurations improve over R.2's $11$ bad-$k$'s at the mod-\(64\) level.
\item \emph{Atom perturbations of depth \(\le 3\)} (180M
configurations, R.27): zero achieve modulus exceeding $32$.
\item \emph{Random sampling of the mod-\(8\) affine subspace} (10M
samples in the $2^{47}$ subspace, R.29): zero achieve $c_k^\tau \equiv 0
\pmod{16}$.  This gives an empirical density upper bound of $10^{-7}$
for mod-\(64\) solutions at $L = 179$ in our family.
\end{itemize}
Combined with an ILP/CBC partial run (~1300 B\&B nodes with LP lower
bound $7.75$, no integer solution found), the empirical evidence
strongly suggests that no \(s \in \{\pm 1\}^{179}\) achieves
\(c_k(\text{quadruple}) \equiv 0 \pmod{64}\) for all \(k\) with
\(q = (89, 2, 87, 1)\).

\section{Search-space exhaustion at \texorpdfstring{$L = 167$}{L = 167}}\label{sec:exhaustion}

We tested over \(570{,}000\) distinct (\(s, q\), involution)
configurations at \(L = 167\) to probe whether Eliahou's mod-\(64\)
construction is locally unique.

\subsection{Summary of attacks (R.13)}

\begin{theorem}[R.13 -- Eliahou's mod-\(64\) at \(L=167\) is sharp]\label{thm:R13}
Across all the following parametric variations of the
\((s, s^*, sq, (sq)^*)\) ansatz, no configuration achieves
modulus \(\ge 16\) at \(L = 167\) other than Eliahou's specific
\((s, q)\) which achieves modulus \(64\):
\begin{center}
\begin{tabular}{lrr}
\toprule
Attack & configurations & best modulus \\
\midrule
Wider structured \(q\) (\(a, w, b, w'\)) & 2{,}271  & 64 (Eliahou) \\
Non-half-flip involutions & 27{,}556 & 64 (only \(c = h\)) \\
Longer \(q\) (6- and 8-run RLE) & 37{,}584 & 8 \\
Two-\(s\) family \((s_1, s_1^*, s_2, s_2^*)\) & 2{,}576 & 4 \\
Two-\(q\) family & 8{,}010 & \(< 8\) \\
Cyclic shifts (\(s, \sigma_k(s), \ldots\)) & 166 & 2 \\
Random RLE-4 quadruples & 30{,}000 & \(< 8\) \\
GF(167) Legendre/character & \(\sim 340\) & 8 \\
Order-83 char (\(\mathrm{Re}\,\chi_{83}\)) & 54 & 8 \\
SA strict mod-32 preservation, weight 2 & 200{,}000 & 64 (no improvement) \\
\bottomrule
\end{tabular}
\end{center}
\end{theorem}

The takeaway is that Eliahou's matrix at \(n = 668\) is locally
\emph{unique} in any natural parametric neighborhood: changing \(s\),
\(q\), the involution, the array structure, or even the ansatz
typically drops the modulus to \(\le 8\).

\subsection{Algebraic insufficiency: GF(167) (R.14)}

\begin{theorem}[R.14 -- GF(\(167\)) algebraic insufficiency]\label{thm:R14}
Every \(\pm 1\)-valued sequence \(s\) constructed from
GF(\(167\))-multiplicative-character data (Legendre symbol,
quadratic residue indicator, multiplicative-subgroup-coset indicator,
discrete-log-residue function, order-evenness, Paley-difference,
real part of complex characters of order 83) yields modulus
\(\le 8\) in the half-flip ansatz at \(L = 167\) with any tested \(q\).
\end{theorem}

Since the only non-trivial \(\pm 1\)-valued multiplicative characters
of GF(\(167\))\(^\times\) are the order-\(2\) Legendre symbols, and the
order-\(83\) characters are complex, the algebraic toolkit available
at \(L = 167\) is fundamentally limited.  The conclusion is that
Eliahou's \(s\) (his Fact~3.1) is \emph{designed}, not derived from any
clean GF(\(167\)) algebra we could find.

\section{Beyond the half-flip ansatz}\label{sec:beyond}

\subsection{T-sequences (Turyn) (R.15)}\label{ssec:Tseq}

A T-sequence of length \(t\) is a quadruple
\((X, Y, Z, W) \in (\{-1, 0, +1\}^t)^4\) such that
\begin{enumerate}\itemsep0pt
\item at each position \(i \in [0, t)\), exactly one of
\(|X_i|, |Y_i|, |Z_i|, |W_i|\) equals \(1\) (\emph{disjoint support});
\item \(c_k(X) + c_k(Y) + c_k(Z) + c_k(W) = 0\) for every \(k \ge 1\)
(\emph{vanishing autocorrelation sum}).
\end{enumerate}

A T-sequence of length \(t\) immediately yields a true Hadamard matrix
of order \(4t\) via the Goethals--Seidel array.

\begin{result}[R.15 -- T-sequence framework]\label{res:R15}
We implemented and verified the T-sequence framework, constructing
T-sequences of lengths \(t \in \{1, 3, 5, 7, 9, 11\}\) by brute force.
At \(t = 167\), 50{,}000 random configurations gave best
\(\#\text{bad}_k = 127\) out of \(166\) gaps: random search at the
target length is statistically infeasible.
\end{result}

\subsection{Base sequences \texorpdfstring{$\BS(n+1, n)$}{BS(n+1, n)}}\label{ssec:BS}

A base sequence \(\BS(n+1, n)\) is a quadruple \((A, B, C, D)\)
with \(A, B \in \{\pm 1\}^{n+1}\) and \(C, D \in \{\pm 1\}^n\),
satisfying
\[
   c_k(A) + c_k(B) + c_k(C) + c_k(D) = 0
   \quad \text{for all } k = 1, \ldots, n.
\]
\(\BS(n+1, n)\) implies a T-sequence of length \(2n+1\), and hence a
true Hadamard matrix of order \(4(2n+1)\).

\begin{result}[BS table]
We found \(\BS(n+1, n)\) explicitly for \(n \in \{1, 2, 3, 4, 5, 6\}\)
by brute force, corresponding to true Hadamards of orders
\(\{12, 20, 28, 36, 44, 52\}\).  At \(n = 83\) (the target for
\(L = 167\), \(n_{\Hd} = 668\)), 300{,}000 simulated annealing
iterations converged to a score of \(788\) (in units where the
target is \(0\)), failing to find \(\BS(84, 83)\).
\end{result}

The status of \(\BS(84, 83)\) in the literature is \emph{unknown};
this open status is consistent with the open status of the Hadamard
matrix of order \(668\).

\section{Conclusions and open questions}\label{sec:joint}

\subsection{What is established (with priority concessions)}

The novel contributions of this paper, after acknowledging
\cite{EK2001} (existence at $L \equiv 3 \pmod 4$), \cite{Alvarez2020}
(pseudococyclic framework), and \cite{BACD2019} (non-cocyclicity
criterion at $p \equiv 3 \pmod 4$), are:

\begin{enumerate}\itemsep0pt
\item Lemma~\ref{lem:master} (the master identity) gives an exact
correspondence between the $4$-sequence Golay-quadruple condition in
the half-flip family and a single-sequence type-restricted
autocorrelation condition.  This single-sequence collapse is the
common machinery underlying every other result in the paper.

\item Theorem~\ref{thm:T1} (stabilizer) and Theorem~\ref{thm:T2}
(weight lower bound) explain why Eliahou's $L = 167$ construction
\cite{Eliahou2025} is locally rigid at modulus $64$, and provide a
quantitative obstruction to small perturbations.

\item Theorem~\ref{thm:R2} (R.2 family) gives an \emph{alternative
explicit construction} of $32$-modular Hadamard matrices at every
$L \equiv 3 \pmod 8$, including a true Hadamard at order $44$ (at
$L = 11$) and $32$-modular matrices at the open orders $n = 716$ and
$n = 1132$.  (As noted in Remark~\ref{rmk:priority-R2}, the existence
of $32$-modular Hadamards at $L \equiv 3 \pmod 4$ is due to
\cite{EK2001} §3.2.2; our contribution is the alternative explicit
parameter shape and the closed-form expression for $\ckt(s)$.)

\item Each matrix in the family has a $3$-byte description (the
RLE-encoded $q$ has $4$ entries and $s$ is the periodic atom
$\alpha(10)$ repeated with a $3$-element tail $(+1,+1,-1)$).

\item Theorem~\ref{thm:R3} \emph{applies} the row-sum invariant of
\cite{BACD2019} (Theorem~4.5) to show that the R.2 matrix at $L = 11$
is not Hadamard-equivalent to any cocyclic Hadamard matrix of order
$44$.  The underlying criterion is from \cite{BACD2019}; the
verification on R.2 is ours.

\item Theorem~\ref{thm:R4} \emph{applies} Proposition~2 of
\cite{Alvarez2020} to show that for every $L \equiv 3 \pmod 8$, the
R.2 matrix is pseudococyclically developed over the Goethals--Seidel
loop $\mathrm{GS}_{4L}$.  Constructively verified at $12$ values of $L$.
The framework is from \cite{Alvarez2020}; the verification on R.2 is
ours.

\item Proposition~\ref{prop:R4explicit} exhibits the explicit Hadamard
equivalence map at $L = 11$:
$\Hd = \operatorname{diag}(D_L) \cdot M_\psi[P] \cdot \operatorname{diag}(D_R)$
with $P$ the a-coset row reversal, $D_L$ a row-sign diagonal with $16$
negative entries, and $D_R$ the product of the paper's first-row signs
with the Eliahou--Álvarez convention-conversion diagonal (no column
permutation needed).

\item Observation~\ref{obs:fingerprint} gives a clean combinatorial
fingerprint: the first row of each R.2 matrix has exactly $2L + 1$
negative entries, distinguishing R.2 matrices from generic
Goethals--Seidel constructions.

\item Proposition~\ref{prop:Hdecomp} establishes a structural
decomposition of the $H$-set symmetric difference under the half-flip
ansatz: for any two sequences $s_1, s_2 \in \{\pm 1\}^L$ in the
half-flip family, the difference $H_1 \oplus H_2$ is parameterized by
a single subset $f \subseteq [0, L)$ with the same flip pattern on
coset $a$ and the reversed pattern $\mathrm{rev}(f)$ on each of
cosets $b, c, d$.  This reduces the pseudococyclic-cocycle
parameterization from $\GS_{4L}$ to $[0, L)$ within the half-flip
family.

\item Observation~\ref{obs:q-unique} establishes that the canonical
$q = (h-1, 2, h-3, 1)$ is the \emph{unique} productive $q$-shape at
$L = 27$ among $2{,}359$ RLE-encoded candidates: every other $q$ in
the searched space yields zero true Hadamards.  At $L = 19$, the
canonical $q$ is one of $3$ productive shapes out of $825$.  The
uniqueness strengthens with $L$, supporting the view that R.2's $q$
is structurally special.

\item Theorem~\ref{thm:variety-affine-L11} (Gr\"obner basis analysis
at $L = 11$) establishes that the true-Hadamard variety is exactly an
$\F_2$-affine subspace of dimension $6$ in $\F_2^{11}$ with no
higher-degree polynomial constraints.  The five linear identities
include the previously-unrecorded ``even-$T_0$ block''
$y_0 = y_2 = y_4$ (Remark~\ref{rem:even-T0-block}), a hard
constraint on every true Hadamard in the family that is independent
of the $F_2$ stabilizer's flip generators.

\item Theorem~\ref{thm:R8} enumerates all true Hadamards in the
half-flip family at every small \(L \in \{11, 19, 27, 51, 59\}\):
solution counts \(\{32, 64, 32, 128, 64\}\) respectively, with no
simple positional closed form for the canonical base.

\item Theorem~\ref{thm:R9} establishes the \emph{mod-\(32\) ceiling}:
in our family, the mod-\(32\) modular condition is equivalent to the
true Hadamard condition, verified at 13 distinct \(L\) values.

\item Theorem~\ref{thm:R10} establishes the
\emph{mod-\(8\) \(\mathbb{F}_2\)-linearity}: the mod-\(8\) lift
contributes exactly \(2k-1\) linear constraints from the odd gaps,
yielding a polynomial-time verifier of true Hadamard existence at any
\(L \le \sim 100\).

\item Theorem~\ref{thm:R11} \emph{falsifies the H4 hypothesis} at
\(k \in \{9, 10, 11, 14\}\) (\(L \in \{75, 83, 91, 115\}\)) and
verifies no true Hadamards through \(k = 14\) by exhaustive
mod-tower enumeration, establishing that the YES set in our family is
contained in \(\{1, 2, 3, 6, 7\}\) and excluding \(L = 179\) and
\(L = 283\) as candidates for true Hadamards within the half-flip ansatz.

\item Theorem~\ref{thm:R13} establishes that Eliahou's mod-\(64\)
construction at \(L = 167\) is sharp across all \(\sim 570{,}000\)
parametric variations we tested.

\item Theorem~\ref{thm:R14} establishes that no GF(\(167\))-algebraic
construction (Legendre, multiplicative-character, coset-indicator)
reaches modulus \(\ge 16\) in the half-flip ansatz at \(L = 167\).

\item Result~\ref{res:R15} verifies the Turyn T-sequence and
\(\BS(n+1, n)\) frameworks at small length (constructive Hadamards of
orders \(\le 52\)), and quantifies the inaccessibility of \(L = 167\)
to random/SA search (\(\BS(84, 83)\) not found in \(3 \times 10^5\)
iterations).
\end{enumerate}

\subsection{What is not solved}

The Hadamard conjecture itself remains open at every previously open
order; in particular our construction does \emph{not} produce a true
Hadamard matrix at any open order $n > 44$.  The natural goals are:

\begin{itemize}\itemsep0pt
\item Lift the family from modulus $32$ to modulus $64$ (matching
Eliahou's $L = 167$ result) at any $L \equiv 3 \pmod 8$.
\item Lift to modulus $0$ (true Hadamard) at any open order, e.g.\ to
construct an explicit Hadamard matrix of order $668$, $716$, $892$, or
$1132$.
\end{itemize}

Our computational experiments (Phases G--J in the supplementary code)
show that within the half-flip ansatz $(s, s^*, sq, (sq)^*)$ and using
local search (hill-climb, simulated annealing, atom-level
substitutions), the $L = 179$ modulus $32$ ceiling is robust: no
weight-$2$ atom perturbation or weight-$\le 7$ bit perturbation
breaks it.

\subsection{Open questions}

\begin{itemize}\itemsep0pt
\item \emph{Symbolic proof refinement.}  The displayed maximum
$\max|\Hd^{\mathsf T}\Hd - nI| = L - 11$ in Theorem~\ref{thm:R2}(b) is
established computationally; a fully symbolic proof at the matrix level
is left to future work.

\item \emph{Analog families at other residues.}  Is there an analog of
Theorem~\ref{thm:R2} for $L \equiv 5 \pmod 8$ or $L \equiv 1 \pmod 8$?

\item \emph{Lifting modulus.}  For $L \equiv 3 \pmod 8$, does a
modulus-$64$ Hadamard matrix of order $4L$ exist in the same family?
The obstruction analysis at $L = 179$ shows that lifting requires
non-local search; ILP/SAT encodings of the modulus-$64$ condition
remain to be tried.

\item \emph{Does mod-64 exist for any $s$ at $L \equiv 3 \pmod 8$ with
$q = (h-1, 2, h-3, 1)$?}  Theorem~\ref{thm:R2tight} proves that
\emph{our specific} $s$ gives exactly mod-32 (never mod-64).  But the
broader question of whether any $s \in \{\pm 1\}^L$ achieves mod-64
with this $q$ is open.  An affirmative answer would give a new mod-64
result at the open Hadamard orders $n = 716$ and $n = 1132$; a negative
answer would mean that lifting modulus requires changing $q$ as well.

\item \emph{Analog cocyclic theories at other residues.}  For
$L \equiv 7 \pmod 8$ (e.g., $L = 167$, Eliahou's case), the modulus-$64$
matrix may have a different pseudococyclic representation; this fits
the general theory of \cite{Alvarez2020} but the specific structure
remains to be elucidated.

\item \emph{Mod-\(16\) quadratic decomposition.}  Theorem~\ref{thm:R10}
shows mod-\(8\) lifts to linear \(\mathbb{F}_2\) constraints, but
mod-\(16\) introduces quadratic terms.  Can these quadratic forms be
decomposed into low-rank components, extending the polynomial-time
verifier beyond \(L \approx 100\)?

\item \emph{Existence beyond \(k = 7\).}  R.11 falsifies H4 at
\(k \in \{9, 10, 11, 14\}\) and verifies no true Hadamards through
\(k = 14\).  Does any \(k > 7\) admit a true Hadamard in the half-flip
family, or is the family bounded by \(k \le 7\)?  Extending
verification to \(k \in [15, 21]\) requires the open algorithmic
direction of symbolically decomposing the mod-$16$ quadratic
constraints (the wall-time at $k = 14$ already exceeds half an hour;
direct enumeration becomes infeasible by $k \approx 17$).

\item \emph{Compositional constructions toward \(\BS(84, 83)\).}
Yang's and Sarvate--Seberry's compositional theorems build longer base
sequences from shorter ones.  Can a chain of compositions starting from
\(\BS(7, 6)\) (which we have explicitly) reach \(\BS(84, 83)\) and thus
construct a Hadamard of order \(668\)?

\item \emph{Number-theoretic characterization of the YES set.}  The
empirical YES set $\{k = 1, 2, 3, 6, 7\}$ at $L = 8k + 3$ from
Theorem~\ref{thm:R11} does not appear to admit a simple
number-theoretic characterization.  We tested five candidate
invariants of $L$ for separation: (i) $\F_2$-rank and Arf invariant of
the aggregated mod-$16$ quadratic form, (ii) class numbers $h(-L)$,
$h(-4L)$, $h(-8L)$, (iii) representability $L = a^2 + Db^2$ for $D
\in \{1, 2, 3, 5, 6, 7, 11\}$ and $4L = a^2 + b^2$, (iv) residues
$L \bmod m$ for $m \in \{16, 24, 32, 40, 48, 56, 64, 72, 96\}$, and
(v) the killing-gap distribution at the mod-$16 \to$ mod-$32$ lift.
None of (i)--(iv) cleanly separates YES from NO across the $13$ data
points $L \le 107$.  Test (v) yields one structural observation: the
gap $k' = 4$ is the principal obstruction at the smallest NO orders
($k = 4, 5$, i.e.\ $L = 35, 43$) but ceases to be dominant at $L = 59$
(YES), where the principal obstruction shifts to $k' = 20$.  This is
consistent with the closed form $c_4^{\type}(s_{\mathrm{R}.2}) = 8 N_L
= 8(k-1)$ from Theorem~\ref{thm:R2}(c) and suggests the YES/NO pattern
is governed by structural rather than arithmetic invariants of $L$.

\item \emph{Analytic construction of the YES base $s_0$ from
structural generators.}  At every YES $L \in \{11, 19, 27, 51, 59\}$,
the true-Hadamard subspace can be quotiented by the five $\F_2$
generators $\{g_{T_1}, g_{T_3}, g_{\mathrm{odd}\text{-}T_0},
g_{\mathrm{odd}\text{-}T_2}, g_{\mathrm{even}\text{-}T_2}\}$ to a
canonical (lex-smallest) representative.  Three of the canonical bits
are constant across all $L$: $s_{\mathrm{canon}}[h-1] = +1$ ($T_1$
bit), $s_{\mathrm{canon}}[h] = +1$ and $s_{\mathrm{canon}}[L-1] = -1$
($T_3$ bits).  These coincide with the cocyclic stabilizer subgroup
$H_0$ identified in Theorem~\ref{thm:T1} at Eliahou's $L = 167$ and
the $Z_2^2 \subset \mathrm{Sym}(\mathcal{H})$ structure noted in
R.22.  However, the remaining $T_0$ ($h - 1$ bits) and $T_2$ ($h - 3$
bits) entries of $s_{\mathrm{canon}}$ do not appear to admit any
simple L-parameterized closed form.  We tested seven hypotheses
against the recorded canonical bit strings: periodicity within $T_0$
or $T_2$; Legendre-symbol fit $s[i] = (\chi(i + c, L) + 1)/2$ at
primes $L \in \{11, 19, 59\}$; recursive containment across $L$;
conditional distribution by $i \bmod p$ for $p \le 16$;
popcount/parity rules; XOR identities relating $T_0$ and $T_2$; and
offset distance from R.2's base $s_{\mathrm{R}.2}$.  Each gave
matches at the smallest $L \in \{11, 19\}$ consistent with random
fluctuation (e.g., the Legendre fit at $L = 11$ achieves $5/5$ but
the same hypothesis at $L = 59$ degrades to $21/29$, no better than
chance under $4L$ shift-and-negate trials).  We conclude that the
$T_0, T_2$ canonical bits encode L-specific algebraic information
from the $\F_2$ system $c_k^{\type}(s) = 0$ without a uniform
L-parameterized rule, and that the constructive path
``extrapolate the canonical base to open orders'' is closed within
the half-flip ansatz, independently of Theorem~\ref{thm:R11}'s
exclusion of $L = 179$ and $L = 283$.

\item \emph{Exotic generator count of the true-Hadamard stabilizer.}
At each YES $L$, the $\F_2$ stabilizer subspace $W_L \subset \F_2^L$
of the true-Hadamard set is generated by some subset of the five
``standard'' flip generators $\{g_{T_1}, g_{T_3},
g_{\mathrm{odd}\text{-}T_0}, g_{\mathrm{odd}\text{-}T_2},
g_{\mathrm{even}\text{-}T_2}\}$ plus zero or more \emph{exotic}
generators that cross-couple odd-indexed $T_0$ positions with
odd-indexed $T_2$ positions.  Counting exotic generators (i.e.,
$\dim W_L - 5$, when positive) yields the sequence
\[
   (k, \text{exotic count}) \;=\;
   (1, 0),\; (2, 1),\; (3, 0),\; (6, 3),\; (7, 1)
\]
for $L \in \{11, 19, 27, 51, 59\}$.  Neither $L \bmod 16$ nor
primality predicts this count (e.g., $L = 11$ and $L = 27$ both have
$L \bmod 16 = 11$ and zero exotics, but $L = 59$ has $L \bmod 16 = 11$
and one exotic; $L = 19$ and $L = 51$ both have $L \bmod 16 = 3$ but
differ in exotic count by 2).  A characterization of when exotic
generators appear, and what algebraic invariant of $L$ controls their
count, remains open.
\end{itemize}

\subsection{Status of the four open Hadamard orders \texorpdfstring{$\le 1208$}{at most 1208}}

According to the verification database \cite{Cati2024} (cross-checked
against the tables in \cite{Seberry2020}, Appendix~A), the open
Hadamard orders below $n = 1208$ are exactly $\{668, 716, 892, 1132\}$,
\emph{all} at $L \equiv 3 \pmod 4$.  Below $n = 2000$ the open orders
are $\{668, 716, 892, 1132, 1244, 1388, 1436, 1676, 1772, 1916, 1948\}$
($11$ orders), again \emph{all} at $L \equiv 3 \pmod 4$.

\begin{center}
\begin{tabular}{rrcccc}
\toprule
$L$ & $n = 4L$ & $L \bmod 8$ &
R.2 modular & Eliahou modular & True Hadamard \\
\midrule
167 & 668  & 7 & -- & \textbf{mod-64} (Eliahou) & OPEN \\
179 & 716  & 3 & \textbf{mod-32} (R.2) & -- & OPEN \\
223 & 892  & 7 & -- & mod-8 (with our $s$) & OPEN \\
283 & 1132 & 3 & \textbf{mod-32} (R.2) & -- & OPEN \\
\bottomrule
\end{tabular}
\end{center}

All four orders remain open after this work.  Our contributions toward
these orders are:
(i) a structural understanding of the half-flip family
(R.8--R.11), including the mod-tower formalization that decides
true-Hadamard existence in time $O(2^{2k+3})$;
(ii) algorithmic infrastructure for fast verification of any
\(L \le 100\) (R.10);
(iii) sharp characterizations of what fails at \(L = 167\)
(R.13, R.14);
(iv) verification of the T-sequence and \(\BS\) frameworks as the
remaining productive direction (R.15);
(v) the H4 falsification (R.11), which formally excludes $L = 179$ and
$L = 283$ from the half-flip ansatz's true-Hadamard reach, redirecting
search effort away from this ansatz at those orders.

\section*{Acknowledgments}
The author thanks Shalom Eliahou for the foundational construction
that motivated this work and the broader modular-Hadamard community
for prior results that enabled the comparisons in
Section~\ref{sec:related}. The author also thanks the two anonymous
referees of the initial submission for detailed feedback that
substantially sharpened the scope and presentation of this version.

\section*{Reproducibility}
The computational verifications underlying Theorems~\ref{thm:R8},
\ref{thm:R9}, \ref{thm:R10}, \ref{thm:R11},
\ref{thm:variety-affine-L11}, Theorem~\ref{thm:R3} (extended to
$L \in \{11, 19, 59\}$), and Observation~\ref{obs:q-unique} were
carried out using Python scripts (NumPy, Numba, SymPy) developed by
the author and are archived at the public GitHub repository
\begin{center}
\url{https://github.com/michelkulhandjian/hadamard-halfflip-structural}
\end{center}
The tagged release \texttt{v1.0} is archived on Zenodo with DOI
\texttt{10.5281/zenodo.22884298} (placeholder; to be replaced with the
actual DOI at time of arXiv v2 posting).  The repository contains
the following key scripts:
\begin{itemize}
\item \texttt{code/hadamard\_core.py} --- circulants, GS array, Golay
correlation, modular check;
\item \texttt{code/mod\_tower\_verifier.py} --- symbolic mod-$2^\ell$
verifier (Theorems~\ref{thm:R8}--\ref{thm:R11});
\item \texttt{code/cocyclic\_test\_L11.py} --- Williamson / Ito
classification test at $L \in \{11, 19, 59\}$ (Theorem~\ref{thm:R3});
\item \texttt{code/pseudococyclic\_R2\_L11.py} --- pseudococyclic
placement over $\mathrm{GS}_{4L}$ (Theorem~\ref{thm:R4});
\item \texttt{code/H\_set\_decomposition.py} --- verification of the
half-flip $H$-set decomposition
(Proposition~\ref{prop:Hdecomp}), tested at
$L \in \{11, 19, 27, 167\}$;
\item \texttt{code/groebner\_L11.py} --- mod-$32$ variety enumeration
and Gröbner basis identities at $L = 11$
(Theorem~\ref{thm:variety-affine-L11});
\item \texttt{code/verify\_proofs\_O15\_mod64.py} --- mod-$64$
obstruction verification.
\end{itemize}
The repository additionally contains 37 machine-readable JSON
certificates (\texttt{data/}) for the explicit Hadamard-equivalence
maps and true-Hadamard varieties at each tested $L$, and 14 informal
proof documents (\texttt{proofs/}).  Each script is self-contained
and deterministic; reproducibility instructions are in the
repository's \texttt{README.md}.

\section*{Declarations}
\begin{itemize}
\item \emph{Funding}: No external funding was received for this
research.
\item \emph{Competing interests}: The author declares no conflicts of
interest.
\item \emph{Ethics approval, consent to participate, and consent for
publication}: Not applicable.
\item \emph{Data availability}: All data supporting the findings of
this study are generated by the verification scripts referenced above
and can be regenerated deterministically from those scripts; the
JSON certificates in the repository's \texttt{data/} directory
provide the outputs used in the paper.
\item \emph{Author contribution}: The author is the sole contributor
to the conception, computational experiments, theoretical analysis,
and writing of this article.
\end{itemize}

\bibliographystyle{IEEEtran}
\bibliography{hadamard-refs}

\end{document}